\documentclass[12pt]{amsart}
\usepackage[colorlinks=true, pdfstartview=FitV, linkcolor=blue, citecolor=green, urlcolor=black,filecolor=magenta]{hyperref}

\usepackage{graphicx, verbatim,amsmath,amssymb}

\newcommand{\aaa}{{\mathcal A}}

\newcommand{\eee}{{\mathcal E}}

\newcommand{\wpe}{weakly pattern equivariant}
\newcommand{\spe}{strongly pattern equivariant}

\newcommand{\R}{{\mathbb R}}

\newcommand{\N}{{\mathbb N}}

\newcommand{\Z}{{\mathbb Z}}
\newcommand{\C}{{\mathbb C}}

\newcommand{\sP}{strongly pattern equivariant}
\newcommand{\wP}{weakly pattern equivariant}
\newcommand{\s}{{\mathfrak s}}
\newcommand{\ns}{nonslip}
\newcommand{\brillig}{rigid}
\newcommand{\pim}{\pi_{max}}
\newcommand{\Rmax}{{\mathcal R}_{max}}
\newcommand{\sm}{A_\sigma}

\newcommand{\m}{\mbox{\Large $\curlywedge$}}
\newcommand{\mo}{M}
\newcommand{\V}{\mathcal V}
\newcommand{\F}{\mathcal F}

\newtheorem{thm}{Theorem}[section]
\newtheorem*{thm*}{Theorem}
\newtheorem{cor}[thm]{Corollary}
\newtheorem{lem}[thm]{Lemma}
\newtheorem{prop}[thm]{Proposition}
\newtheorem{dfn}[thm]{Definition}

\everymath{\displaystyle}

\theoremstyle{definition}

\newtheorem{remarks}{Remark}

\begin{document}
\title{Conjugacies of Model Sets}
\author{Johannes Kellendonk and  Lorenzo Sadun}
\address{Johannes Kellendonk\\ Universit\'e de Lyon,
Universit\'e Claude Bernard Lyon 1\\
Institute Camille Jordan, CNRS UMR 5208\\  69622 Villeurbanne, France}\email{kellendonk@math.univ-lyon1.fr}
\address{Lorenzo Sadun\\Department of Mathematics\\The University of
 Texas at Austin\\ Austin, TX 78712 USA} 
\email{sadun@math.utexas.edu}
\thanks{
The work of the first author is partially supported by the ANR SubTile
NT09 564112. 
The work of the second author is partially supported by NSF
  grant  DMS-1101326} 
\date{January 13, 2015}



\begin{abstract}
  Let $M$ be a model set meeting two simple conditions: (1) the
  internal group $H$ is $\R^n$ (or a product of $\R^n$ and a finite
  group) and (2) the window $W$ is a finite union of disjoint
  polyhedra.  Then any Delone set with finite local complexity (FLC)
  that is topologically conjugate to $M$ is mutually locally
  derivable (MLD) to a model set $M'$ that has the same internal group
  and window as $M$, but has a different projection from $H \times
  \R^d$ to $\R^d$.  In cohomological terms, this means that the group
  $H^1_{an}(M,\R)$ of asymptotically negligible classes has dimension $n$.
  We also exhibit a
  counterexample when the second hypothesis is removed, constructing
  two topologically conjugate FLC Delone sets, one a model set and the
  other not even a Meyer set.
\end{abstract}

\maketitle

\setlength{\baselineskip}{.6cm}

\section{Introduction and statement of results.}

A substantial part of the analysis of Delone sets (or tilings) is
based on the study of their associated dynamical systems.
This includes characterizing certain classes of
Delone sets by ergodic and topological properties of their dynamical systems. 
Two Delone sets
  $\Lambda$ and $\Lambda'$ are called {\em topologically conjugate\/}
  whenever their associated dynamical systems are topologically
  conjugate. We call them {\em pointed topologically conjugate\/} if
  the conjugacy maps $\Lambda$ to $\Lambda'$.  We consider which properties of 
finite local complexity (FLC)
  Delone sets are preserved under topological conjugacy or pointed
  topological conjugacy.

In this paper we concentrate on Meyer sets, and specifically on model sets.
Recall that a Meyer set is a Delone set $\Lambda$ such that the set
of difference vectors $\Lambda-\Lambda$ is uniformly discrete, and
that model sets are Meyer sets arising from a particular
cut and project construction. (See Section \ref{sec-model} for precise definitions.) 
This construction involves a locally compact abelian group $H$, a lattice
$\Gamma\subset H\times\R^d$ and a strip $S=W\times \R^d$ where
$W\subset H$, the so-called window, is a compact subset that is the
closure of its interior.  The projection set for these data is the set
of points arising by projecting the points of $S\cap \Gamma$ onto
$\R^d$ along $H$.  We say that a Delone set $\mo'$ is a {\em reprojection\/} 
of a model set
$\mo$ if it arises from the same setup, except that the projection of
$S\cap\Gamma$ onto $\R^d$ is not along $H$, but is along a subgroup
$H' \subset
H\times \R^d$  that is isomorphic to $H$ and transverse to $\R^d$. 
(See Section \ref{sec-model} for more details.)

In \cite{KS13} we showed
that the Meyer property is {\em not\/} always 
preserved under topological conjugacy. In
view of this, we call a Meyer set $\Lambda$ {\em rigid\/}
if every Delone set topologically conjugate to $\Lambda$ is a
Meyer set. Our aim in this article is to study rigidity for
model sets.


We pay particular attention to {\em polyhedral\/} model sets, by which
we mean model sets satisfying two additional assumptions:
\begin{enumerate}
\item[H1.] The internal group $H$ is a vector space
$\R^n$.\footnote{This assumption can be relaxed somewhat. Everything
that we prove about polyhedral model sets also applies
when $H$ is the product of $\R^n$ with a finite group
$C$. A famous example is the Penrose tiling, with $H=\R^2 \times \Z_5$. 
In such a setting, the model set will be a finite union of disjoint polyhedral
model sets.}
\item[H2.] The window $W$ is a finite union of polyhedra. 
\end{enumerate}
Our main result states that such model sets are extremely rigid:

\begin{thm}\label{NewMainTheorem} 
If $\mo$ is a polyhedral model set 
and $\Lambda$ is a Delone set that is pointed topologically conjugate to
$\mo$, then $\Lambda$ is mutually locally derivable (MLD) to a 
reprojection of $\mo$.  
\end{thm}
Without hypothesis H2, this theorem is false.  In Section
\ref{examples} we exhibit a one dimensional model set $\mo$
satisfying H1 but not H2, and a Delone set $\Lambda$ that is pointed
topologically conjugate to $\mo$, with $\Lambda$ not being a
Meyer set, much less a model set or a reprojection of $\mo$.
\medskip

A first result in the direction of studying topological conjugacies
can be found in \cite{KS13}. It says that any pointed topological
conjugacy between repetitive FLC Delone sets is the composition of an
MLD transformation followed by a {\em shape conjugacy\/} (defined below)
which can be chosen arbitrarily close to the identity.  Given that MLD
transformations are well-understood, this reduces the task of
understanding topological conjugacies to a study of shape conjugacies.
Shape conjugacies modulo MLD transformations are parametrized (at
least infinitesimally) by a subgroup $H^1_{an}(\Lambda,\R^d)$
of the first tiling cohomology
with $\R^d$-coefficients $H^1(\Lambda,\R^d)$, called the {\em
  asymptotically negligible group\/} in \cite{CS2}.

If $\Lambda$ is a Meyer set, then within $H^1_{an}(\Lambda, \R^d)$
there is a subgroup of shape deformations that preserve a property closely
tied to the Meyer
property. We call these {\em \ns\/}, and denote the subgroup
$H^1_{ns}(\Lambda, \R^d)$. If $\Lambda$ is a model set, there is a
further subgroup, denoted $H^1_{repr}(\Lambda, \R^d)$, corresponding
to reprojections. One can similarly define $H^1_{an}(\Lambda, \R)$,
$H^1_{ns}(\Lambda, \R)$, and $H^1_{repr}(\Lambda, \R)$. (See Sections
\ref{sec-define} and \ref{sec-tove}, below.)  In cohomological terms,
Theorem \ref{NewMainTheorem} can be restated as follows:

\begin{thm} \label{MainTheorem2} If $\mo$ is a polyhedral model set 
  then $H^1_{repr}(\mo,\R)=H^1_{ns}(\mo, \R)=H^1_{an}(\mo, \R)$.
\end{thm}

Another cohomological restatement is as follows: There is a natural map 
from the cohomology of the quotient space
$(H\times\R^d)/\Gamma$ to the cohomology of any model set $\mo$ constructed 
from the data 
$(\Gamma,H,\R^d)$ (with arbitrary window $W$). 
Let us denote its image by $H^1_{max}(\mo,\R)$ 
\begin{thm} \label{MainTheorem3} If $\mo$ is a polyhedral model set 
  then $H^1_{an}(\mo,\R)$ is $n$-dimensional and is contained 
in $H^1_{max}(\mo,\R)$.
\end{thm}

The organization of the paper is as follows. In Section
\ref{sec-define} we review the machinery of Delone dynamical systems
that is needed in the remainder of the paper. 
In Section \ref{sec-model} we review the theory of model sets, identifying how
different model sets with the same parameter can differ.
In Section
\ref{sec-tove} we introduce the notion of \ns{} generators of shape conjugacies
and show
(Theorem \ref{mainlemma}) that all \ns{} generators are, up to local
deformation, generators of reprojections. In Section \ref{brillig} we
show (Theorem \ref{PolyhedralTheorem}) that asymptotically negligible
classes are represented by coboundaries of \ns{} generators. Taken
together, this proves Theorem \ref{NewMainTheorem}.  In Section
\ref{sec-cohomology} we interpret these results in terms of cohomology
and prove Theorems \ref{MainTheorem2} and \ref{MainTheorem3}.

Nonslip generators are introduced as a means of proving Theorem
\ref{NewMainTheorem}, but we believe that they have independent
interest. In Section \ref{sec-Meyer} we explore the significance of
the \ns{} property for model sets that do not necessarily satisfy
hypotheses H1 and H2, and for more general Meyer sets.  We prove
\begin{thm}\label{MainTheorem4}
Let $\Lambda$ be a repetitive Meyer set and $F$ a generator of
a shape deformation. If $F$ is not \ns{}, then the deformed set $\Lambda^F$
is not Meyer.
\end{thm}
In Section \ref{examples} we exhibit a model set that does not satisfy
H2 and a generator of shape conjugacies that is not \ns{}, and a
corresponding deformation of a model set that is not Meyer. We do not
know whether it is ever possible to construct a \ns{} class that is
not a reprojection.


\section{Preliminaries on point sets and their dynamical
  systems}\label{sec-define}
In this section we review some of the necessary background on Delone
sets and their dynamical systems.  

\subsection{Dynamical system of a Delone set}
A {\em Delone set\/} is a set $\Lambda \subset \R^d$ that is uniformly
discrete and relatively dense. That is, there exists an $r>0$ such
that every ball of radius $r$ contains at most one point of $\Lambda$,
and there exists an $R>0$ such that every ball of radius $R$ contains
at least one point of $\Lambda$. 
A {\em Meyer set\/} is a Delone set $\Lambda$ for which
$\Lambda-\Lambda$ (i.e., the set of displacement vectors between
points of $\Lambda$) is uniformly discrete.  

Let $B$ be a compact subset of $\R^d$.
The $B$-{\em patch\/} of a point set $\Lambda\subset \R^d$ is the intersection $P=\Lambda\cap B$ 
of $\Lambda$ with $B$. 
We denote it by $(P,B)$ or simply by $P$. An $R$-patch of $\Lambda$ at $x$ is 
the intersection of $\Lambda$ with $B=B_R(x)$, the 
closed ball of radius $R$ centered at $x$. 

A Delone set has {\em finite local complexity\/}, or FLC, if for each
$R>0$ the set $\{B_R(0)\cap(\Lambda-x)|x\in\Lambda\}$ is finite. That is, 
if the number of $R$-patches occuring at points of
$\Lambda$ and counted up to translation is finite.  A Delone set is
{\em repetitive\/} if for every patch $P$ of $\Lambda$, there exists an
$R$ such that every $R$-patch of $\Lambda$ contains at least one
translated copy of $P$. {\em All Delone sets considered in this paper
  have FLC and are repetitive.}

Delone sets are associated with dynamical systems, an idea which goes back to \cite{Rudolph} (see, for instance \cite{ABKL} for a recent description). We pick
a metric on the space of Delone sets with given inner and outer radii
$r$ and $R$ such that two Delone sets are close if their restriction
to a large ball around the origin agree exactly, 
up to a small translation.\footnote{Some authors merely require the restrictions
of the Delone sets to the large ball to be close in the Hausdorff metric. 
If the Delone sets have FLC, this yields the same topology 
as our definition. However, when working with Delone sets that do not 
have FLC, the two topologies are different.}
$\R^d$ acts on the
space of Delone sets by translation. The closure of the orbit of a
Delone set $\Lambda$ is called the {\em continuous hull\/} of
$\Lambda$, and is denoted $\Omega_\Lambda$, or just $\Omega$ when
there is no ambiguity about which Delone set is being considered. 
$\Omega_\Lambda$ is compact if and only if $\Lambda$ has FLC. As a 
result, if $\Lambda$ has FLC and $\Omega_{\Lambda'}$ is homeomorphic to 
$\Omega_\Lambda$, then $\Lambda'$ also has FLC.

If
$\Lambda$ is a repetitive FLC Delone set, then $(\Omega_\Lambda,
\R^d)$ is a minimal dynamical system. We will also consider the
{\em canonical transversal\/} $\Xi_\Lambda$ (or simply $\Xi$) of
$\Omega_\Lambda$ which is given by the closure of the set
$\{\Lambda-x:x\in\Lambda\}$.  $\Xi_\Lambda$ consists of all point
patterns of $\Omega_\Lambda$ that contain the origin.

A Delone set $\Lambda'$ is {\em locally derived\/} from $\Lambda$ if
there exists a radius $R>0$ such that, whenever $\Lambda-x_1$ and
$\Lambda-x_2$ agree to radius $R$ around the origin, $\Lambda'-x_1$
and $\Lambda'-x_2$ agree to radius 1 around the origin. If $\Lambda'$
is locally derived from $\Lambda$ and $\Lambda$ is locally derived
from $\Lambda'$, we say that $\Lambda$ and $\Lambda'$ are {\em
  mutually locally derivable\/}, or MLD.

A local derivation of $\Lambda'$ from $\Lambda$ extends to a factor
map from $\Omega_\Lambda$ to $\Omega_{\Lambda'}$. If $\Lambda$ and
$\Lambda'$ are MLD, then this factor map is a topological conjugacy
called an {\em MLD map\/} \cite{CS2}.

\subsection{Maximal equicontinuous factor}
An important factor of the dynamical system $(\Omega,\R^d)$ is the
largest factor (up to conjugacy) on which the action is
equicontinuous. We denote this so-called {\em maximal equicontinuous
  factor\/} by $\Omega_{max}$ and the factor map by $\pi_{max}$.

The equivalence relation
$$\Rmax = \{(\Lambda_1,\Lambda_2)\in \Omega_\Lambda\times\Omega_\Lambda:
\pi_{max}(\Lambda_1) = \pi_{max}(\Lambda_2)\}$$ will play a significant role
in understanding nonslip generators. The following properties were proven
in \cite{BargeKellendonk}:
\begin{lem}

\begin{itemize}
\item Any topological conjugacy preserves $\Rmax$.
\item If two elements $\Lambda_1$, $\Lambda_2$ in the hull of a Meyer 
set satisfy
$\pi_{max}(\Lambda_1) = \pi_{max}(\Lambda_2)$, then they share a point.
\end{itemize}
\end{lem}

\section{Model sets}\label{sec-model}
A model set is a Meyer set that is
obtained by a particular construction. 

\subsection{Cut  and project scheme}
To construct a model set one needs a {\em cut
  and project scheme\/}
$(\Gamma,H,\R^d)$ and a subset $W\subset H$ called the {\em window}.

The cut \& project scheme consists of the space $\R^d$ in which the
model set lives (the {\em parallel\/} or {\em physical\/} space), a
locally compact abelian group $H$ (called the {\em internal group\/} or
{\em perpendicular space\/}) and a lattice (a cocompact discrete
subgroup) $\Gamma\subset H\times\R^d$. The set $S:=W\times \R^d$ is
called the {\em strip\/}.  As usual, we require three further
assumptions:
\begin{enumerate}
\item The projection onto the second factor $\pi^\|:H\times \R^d\to \R^d$ is 
injective when restricted to the lattice $\Gamma$,
\item Projection onto the first factor $\pi^\perp:H\times \R^d\to H$ maps 
the lattice $\Gamma$ densely into $H$,
\item If $W + h = W$ for $h\in H$ then $h=0$.
\end{enumerate}
We also use the notation $\Gamma^\|=\pi^\|(\Gamma)$,
$\Gamma^\perp=\pi^\perp(\Gamma)$, $x^\|=\pi^\|(x)$ and $x^\perp=\pi^\perp(x)$.
We write $\pi_\Gamma^\perp$ for the restriction of $\pi^\perp$ to $\Gamma$. The point set
$$\m_\xi(W) := \{\pi^\|(\gamma): \gamma \in S\cap(\Gamma+\xi)\}$$
is called the projection set of the cut \& project scheme with window
$W$ and parameter $\xi\in H\times \R^d/\Gamma$. 
Note that any element of $\ker\pi_\Gamma^\perp$ is a period of the projection set $\m_\xi(W)$.

We are interested in windows $W$ that are compact and the closure of
their interiors. In this case, the projection set $\m_\xi(W)$ is repetitive
if the parameter $\xi$ is such that 
$\pi^\perp(\Gamma + \xi) \cap \partial W$ is empty. We call such
parameters {\em non-singular\/} and denote the 
set of non-singular parameters by $NS$. 

\begin{dfn}
  A {\em model set\/} (with window $W$) is an element in the hull of a
  projection set $\m_\xi(W)$ whose parameter is non-singular. We
  always assume that $W$ is compact and the closure of its interior.
  We call the model set {\em polyhedral\/} if $H = \R^n$
  and the window is a finite disjoint
  union of polyhedra.
\end{dfn}
Model sets are repetitive Meyer sets.
It is well-known that MLD maps send model sets
to model sets, and send Meyer sets to Meyer sets.

\begin{remarks}
  Our use of the term {\em model set\/} is slightly different than
  that of some other authors.  First, a model set in our sense need
  not to be a projection set with closed window, but might be what
  elsewhere is called an {\em inter-model set\/}, containing some but
  not all points $x^\|$ for which $x^\perp$ lies on the boundary of
  $W$. Second, the requirement that the parameter be non-singular
  automatically makes our model sets repetitive. 
We also note that the window being
  the closure of its interior is not required by all authors. For
  instance \cite{BaakeGrimm13} demand only that the window be
  relatively compact and have non-empty interior.
\end{remarks}

\subsection{Non-singular parameters and model sets}
A model set is called {\em non-singular\/} if it is a projection set
$\m_\xi(W)$ with non-singular parameter $\xi$. 
Since the window is
the closure of its interior, this occurs for a dense $G_\delta$-set of
$\xi$. Therefore these model sets are also
called {\em generic\/}.
   
For a fixed window $W$, the hull of a non-singular model set $\m_{\xi_0}(W)$ contains
all other model sets $\m_\xi(W)$ with $\xi \in NS$. In other words,
the hull of a non-singular model set with window $W$ depends not on
the choice of the non-singular parameter $\xi$ but only on the
cut and project scheme and the window. 

Since the cut and project scheme remains largely
the same in what follows we denote the hull by $\Omega(W)$. Hence a model
set with window $W$ is an element of $\Omega(W)$.

It is not difficult to see that $\m_{\xi}(W)=\m_{\xi'}(W)$ if
and only if $\xi-\xi'\in\Gamma$. Thus to every model set 
$\m=\m_\xi(W)$ with non-singular $\xi$, we can uniquely associate the
set $\xi + \Gamma$ of parameters associated to $\m$. This map from 
the nonsingular model sets onto $NS/\Gamma$ can be extended 
by continuity to a map from 
$\Omega(W)$  onto the closure of $NS/\Gamma$. Since the non-singular parameters
are dense, the image of this map is the torus 
$(H\times \R^d)/\Gamma$; the map is called the 
{\em torus parametrization}. 
It turns out that this torus is also 
the maximal equicontinuous
factor, the action being given
by left translation on the second factor $\R^d$
\cite{BargeKellendonk}. 

To recap: 
The maximal equicontinuous 
factor map $\pim:\Omega(W)\to (H\times \R^d)/\Gamma$ is the continuous
extension of the map $\m_\xi(W)\mapsto \xi$ (for $\xi\in NS$)
which associates to a non-singular model set
$\m_\xi(W)$ its parameter $\xi$ (mod $\Gamma$). 
This map is injective precisely on
its pre-image of $NS$.
The elements of the hull that are mapped
by $\pim$ to singular points are called {\em singular model sets\/}. 

If $M$ is a singular model set, we will call $\pim(M)$ the parameter of $M$.
$M$ is not generally equal to the projection set $\m_{\pim(M)}(W)$. However, 
\begin{prop}\label{prop-nonsing}
Let $M$ be a model set with window $W$ and $\xi = \pim(M)$. Then 
$$ \m_\xi(Int(W))\subset M \subset \m_\xi(W).$$
\end{prop}
\begin{proof}
  A model set $M$ with window $W$ is a limit of a sequence $(\m_\eta
  -x_n)_n$ where $\eta$ is non-singular and $x_n\in \R^d$. But
  $\m_\eta - x_n = \m_{\xi_n}$, where $\xi_n = \eta +
  [0,x_n]_\Gamma$. By continuity of $\pim$ we thus have $\xi = \lim
  \xi_n$ and so the sequence $(\xi_n)_n$ lifts to a sequence
  $(\tilde\xi_n)_n\subset H\times\R^d$ which converges to a lift
  $\tilde\xi$ of $\xi$. We may quickly restrict to the case that
  $\tilde\xi_n^\|=0$, because otherwise we can replace $x_n$ by
  $x_n-\tilde\xi_n^\|$ and $\eta$ by $\eta-[0,\tilde\xi^\|]_\Gamma$ to
  reduce to that situation. Then $M$ and all $\m_{\xi_n}$ are subsets
  of $\Gamma^\|$.  In particular, $\Gamma^\|$ is the common domain of
  all functions $\sigma_{\xi_n}$, and the sequence of functions
  $(\sigma_{\xi_n})_n$ converges uniformly to $\sigma_{\xi}$, since
  $(\tilde\xi_n)_n$ converges to $\tilde\xi$. Hence if
  $\sigma_{\xi}(x)\in Int(W)$ then an open neighborhood of
  $\sigma_{\xi}(x)$ in $Int(W)$ contains all $\sigma_{\xi_n}(x)$ for
  $n$ sufficiently large which shows that $\sigma_{\xi_n}(x)\in W$ for
  all $n$ sufficiently large. The latter means that $x$ belongs to
  $\m_{\xi_n}$ for all $n$ sufficiently large and thus also to
  $M$. This shows that $\m_\xi(Int(W))\subset M$.

  To obtain the other inclusion we use the same kind of argument but
  for the complement $W^c$ instead of $Int(W)$.  Indeed, if
  $\sigma_{\xi}(x)\in W^c$ then an open neighborhood of
  $\sigma_{\xi}(x)$ in $W^c$ contains all $\sigma_{\xi_n}(x)$ for $n$
  sufficiently large which shows that $x$ cannot belong to $M$.
\end{proof}

\subsection{The star map}
The star map sends points in $\R^d$ to points in $H$. There are several related star maps to be considered:
\begin{itemize}
\item The {\em general star map\/} $\sigma$ is a group homomorphism
  $\Gamma^\| \to H$, $\sigma(x) = \pi^\perp(\gamma)$ where
  $\gamma\in\Gamma$ is the unique lift of $x\in\Gamma^\|$ in $\Gamma$
  under $\pi^\|$. In particular, $[(\sigma(x),x)]_\Gamma =
  [(0,0)]_\Gamma$, where $[\cdot]_\Gamma$ denotes an equivalence class
  mod $\Gamma$. We denote the general star map also simply with a
  star, $\sigma(x)=x^*$.  The general star map is not continuous if
  one gives $\pi^\|(\Gamma)$ the relative topology induced by $\R^d$.

\item The {\em general star map $\sigma_\xi$ with parameter $\xi$\/}
  sends $\pi^\| (\Gamma + \xi)$ to $H$, where $(\sigma_\xi(x), x)$ is
  the unique lift of $ x \in \pi^\| (\Gamma + \xi)$ to $\Gamma+\xi$.
  Like the general star map without parameter (or equivalently, with
  parameter 0), this has a dense domain in $\R^d$, dense range in $H$,
  and is not continuous.

\item If $\Lambda$ is a Delone subset of $\pi^\| (\Gamma + \xi)$ we
  denote the restriction of $\sigma_\xi$ to $\Lambda$ by
  $\sigma_\Lambda$ and call it the {\em star map of $\Lambda$\/}.  Note
  that the support of $\sigma_\Lambda$ is uniformly discrete. This
  will allow us later to talk about weak pattern equivariance of
  $\sigma_\Lambda$. If $\Lambda$ is a model set with window $W$ then
  the image of $\sigma_\Lambda$ is a dense subset of $W$.
\end{itemize}
Note that if $x_1$ and $x_2$ are points of $M$ and $\pim(M) = \xi$,
then both $(\sigma_M(x_1), x_1)$ and $(\sigma_M(x_2),x_2)$ are in
$\Gamma + \xi$, so $x_2-x_1\in\Gamma^\|$ and
$$\sigma_M(x_2) - \sigma_M(x_1) =\sigma_\xi(x_2) - \sigma_\xi(x_1) = (x_2 - x_1)^*.$$
The factor map $\pim$ is related to the star map of a pattern as
follows. If $M$ is a model set and $\xi = \pim(M)$ then
$(\sigma_M(x),x) \in \Gamma + \xi$ for all points $x\in M$. Thus
\begin{equation*}
\pim(M)= [(\sigma_M(x),x)]_\Gamma.
\end{equation*} 
In particular, if $0 \in M$, then $\pim(M)= [(\sigma_M(0),0)]_\Gamma$.
Let $\imath_{H}:H\to H\times\R^d/\Gamma$ be given by $\imath_{H}(h) =
[h,0]_\Gamma$. By assumption $\imath_{H}$ is injective. Furthermore,
by equivariance of $\pim$ we have $\pim(M-x)= [(\sigma_M(x),0)]_\Gamma$
for all $x\in M$. Thus
\begin{equation}\label{eq-pistar}
\sigma_M(x) = \imath_H^{-1}\circ \pim(M-x) 
\end{equation}
for all $x\in M$.

\subsection{Acceptance domains of patches}
Let $M$ be a model set with window $W$
and let $B$ be a compact set. We call the set
$P= M \cap B$ the ``$B$-patch of $M$''.  More
generally we say that a finite set $P\in\R^d$ is a patch for a model set with window $W$
if there is $M\in\Omega(W)$ and a compact set $B\in\R^d$ such that $P
= M\cap B$. 
We wish to determine a condition for when this is the case.
It will be sufficient for our applications to consider the case that
$P$ contains the origin which we will assume throughout.

Let $D$ be the preimage $\sigma^{-1}((W-W)\cap\Gamma^\perp)$. 
This gives the set of all possible displacements between points in the same pattern, since if 
$x_1, x_2 \in M$, then $\sigma_M(x_2) - \sigma_M(x_1)=\sigma(x_2-x_1) \in W-W$.   
Let $P' = (D \cap B) \backslash P$. $P'$ is
the set of points that can appear in a $B$-patch of some model set which contains $0$
but are not in the $B$-patch $P$. Note that $P'$ is a finite set. Let 
$$W_P^o = {\bigcap_{x \in P} (Int(W) - {x}^*) \cap \bigcap_{x' \in P'} (W^c - {x'}^*)}.$$
$W_P:=\overline{W_P^o}$ 
is called the {\em acceptance domain\/} of $P$. By construction, it is
the closure of an open set and hence the closure of its interior. 

\begin{prop}\label{prop-acc}
Let $P$ be a $B$-patch for a model set with window $W$.
We assume that $P$ contains the origin. Let
$M\in\Omega(W)$ be a possibly different model set which contains the origin. 
If $M$ is non-singular then, for all $x\in M$,  $\sigma_{M} (x)
\in Int(W_P)$ if and only if $P = (M-x)\cap B$. 
\end{prop}

\begin{proof}
Let $\xi=\pim(M)$. Assuming that $M$ is
non-singular this means that $M =\m_\xi(W)$.
The condition $M \cap B= P$ has two parts:
\begin{enumerate}
\item All the points of $P$ should be in $M$. This is equivalent to
  having $\sigma_{\xi}(x) \in W$ for all $x\in P$. But
  $\sigma_{\xi}(x)-\sigma_{\xi}(0) = x^*$, so this is in turn
  equivalent to $\sigma_{M}(0)\in W - x^*$.
\item All the points of $P'$ should not be in $M$. That is, for each
  $x' \in P'$, $\sigma_{\xi}(x') \in W^c$, so $\sigma_{M}(0) \in W^c
  - {x'}^*$, where $W^c$ denotes the complement of $W$.
\end{enumerate}
Thus $P \in M$ if and only if 
\begin{equation} \label{accept} 
\sigma_{M}(0) \in \bigcap_{x \in P} (W - {x}^*) \cap \bigcap_{x' \in P'} (W^c - {x'}^*).
\end{equation} 
Recall that $\pim(M) = [(\sigma_{M}(0),0)]_\Gamma$. Therefore and
since $M$ was assumed non-singular, 
$\sigma_{M}(0)$ cannot lie on the boundary of $W_P^o$. Thus  
we can replace the right-hand side with the closed set $W_P$.
\end{proof}

\begin{lem}
  Suppose that $M$ is a model set containing the origin and satisfying
  H1 and H2. The acceptance domain of every (non-empty) patch of $M$
  containing the origin can be written as a finite union of closed
  convex sets that have non-empty interior.
\end{lem}
\begin{proof} Since $W$ is a finite union of connected polyhedra this
  is also the case for $W_P$.  We can decompose $W_P$ into a finite
  union of connected polyhedra and each of those into a finite union
  of convex polyhedra. Since $W_P$ is the closure of its interior the
  convex polyhedra can be taken to have non-empty interior.
\end{proof}
We formulate the result of the above lemma as a hypothesis on the
window $W$ that is weaker than H2.
\begin{enumerate} 
\item[H$2'$.] The acceptance domain of every patch containing the origin
can be written as a finite union of closed convex sets that have non-empty interior.
\end{enumerate}

\begin{lem}\label{lem-W} Consider a cut \& project scheme
  $(\Gamma,H,\R^d)$ and a compact subset $W\subset H$ that is the
  closure of its interior.
  For any neighborhood $U$ of $0\in H$ there exists a finite set
  $J\subset W\cap\pi^\perp(\Gamma)$ such that
  $\emptyset\neq\bigcap_{u\in J} W-u \subset U$.
\end{lem}

\begin{proof}
  Without loss of generality we can assume that $U$ is contained in
  the compact set $W-W$, so that $K = (W-W) \backslash U$ is compact.
  Since $W \cap \pi^\perp(\Gamma)$ is countable, we can find a
  sequence of nested finite sets $J_1 \subset J_2 \subset \cdots$ with
  $\bigcup J_i = W \cap \pi^\perp(\Gamma)$.  If every intersection
  $\bigcap_{u\in J_i} W-u$ contains a point $x_i \in K$ then, by
  compactness, there is a limit point $x_\infty \in K$ such that
  $x_\infty \in \bigcap_{u\in \pi^\perp(\Gamma)} W-u$.  But
  Schlottmann proved \cite{Schlottmann}[Lemma~4.1] that $\bigcap_{u\in
    W \cap \pi^\perp(\Gamma)} W-u = \{0\}$, which is a contradiction.
\end{proof}

\begin{cor}\label{cor-W} Let $M$ be a model set containing the origin.
  Any open subset $U$ of the window of $M$ contains the acceptance domain
  for a patch of $M$ that contains the origin.
\end{cor}

\begin{proof} 
Pick $x \in M$ such that $\sigma_M(x) \in U$, and let 
$U' = U - \sigma_M(x)$, which is an open neighborhood of $0$. 
We apply Lemma~\ref{lem-W} to obtain a finite set $J\subset
W\cap\pi^\perp(\Gamma)$ such that 
  $\emptyset\neq\bigcap_{u\in J} W-u \subset U'$. For each $u$ pick
$q\in M$ such that $\sigma_M(q) -\sigma_M(x) = \sigma(q-x) =
u$. Denoting by $Q$ the set of such  points $q$ we have 
$\emptyset \neq \bigcap_{q \in Q} W - \sigma(q)\subset U$. 
Let $B\subset \R^d$ be any compact neighborhood of the origin containing $Q$. 
The acceptance domain of $P = B\cap M$ is then a subset of $U$. Furthermore, $P$ contains the origin.
\end{proof}

\subsection{Singular model sets}
We now describe, in more detail than in 
Proposition \ref{prop-nonsing}, 
the potential difference between two singular model
sets of the same hull that have the same parameter. We keep the hull fixed,
and hence the cut and project scheme and window $W$, 
but we allow the parameter to vary.
We assume for our analysis that the window $W$
is polyhedral and decompose it as a polyhedral complex. 

Let $\F(W) $
be the set of open faces of $W$.  Thus $W$ is the disjoint union of
its interior $Int(W)$ with the $f\in\F(W)$.
Let $V(f)$ be the vector space parallel to
$f$, that is, the space spanned by $f-f$, and $A(f) = V(f) +f$ the affine space parallel to $V(f)$
which contains $f$. The closure of $V(f)\cap \Gamma^\perp$ can be written
$$ \overline{V(f)\cap \Gamma^\perp} = H(f) + \Delta(f) $$
where $H(f)$ is the connected component of $0$ of this closure, 
and hence a real vector space, and
$\Delta(f)$ a (uniformly) discrete subgroup of $\Gamma^\perp$. 
Note that $\sigma^{-1}(H(f)\cap \Gamma^\perp) $ is a sublattice of $\Gamma^\|$; we
let $E(f)$ be its real span. 
\begin{lem}\label{lem-codim} 
$E(f)$ has codimension at least $1$. 
\end{lem}
\begin{proof}
We may assume that $H(f)$ is non-trivial.
Let $w$ be a bounded open subset of $H(f)$.
Then $\m_0(w)$ 
is the projection set with window $w$ and parameter $0$
for the cut \& project scheme
$({\pi_{\Gamma}^\perp}^{-1}(H(f)\cap\Gamma^\perp),H(f),E(f))$. Since $w$ is open 
$\m_0(w)$ is relatively dense in $E(f)$ \cite{BaakeGrimm13} and thus has strictly positive lower density in $E(f)$. 

Since $\Gamma^\perp$ is dense in $H$ we can find an open subset
$w\subset H(f)$ and an infinite
subset $\Psi\subset \Gamma^\perp$ such that $w+\psi\subset W$ for
all $\psi\in\Psi$ and such that the sets $w+\psi$ have pairwise empty
intersection. Hence $\m_0(W)$ contains the disjoint union of all
$\m_0(w+\psi)$, $\psi\in\Psi$. Now the lower density of
$\m_0(w+\psi)$ in $E(f)$ is independent of $\psi$. Therefore, if
$E(f)$ has dimension $d$ and hence is all of $\R^d$ than the lower
density of $\m_0(W)$ must be infinite, which is a contradiction.
\end{proof}

Choose a parameter $\xi$ such that 
$A(f)\cap (\Gamma+\xi)^\perp$ is not empty. Then its closure is of the form 
$\overline{A(f)\cap (\Gamma+\xi)^\perp} = H(f) +
\Delta(f) + h(f)$ where $h(f)\in (\Gamma+\xi)^\perp$. The vector $h(f)$ is of course not uniquely determined by this splitting and we have to make a choice, but this choice can be made independent of the choice of $\xi$. 
Indeed if $A(f)\cap (\Gamma+\xi')^\perp$ is also not empty then $\xi^\perp-{\xi'}^\perp\in H(f)+\Delta(f)$. 
Given that $f$ is bounded, there is a 
finite subset $\Phi(f)\subset \Delta(f)+h(f)$ 
such that, for any choice of parameter $\xi'$ (even if the l.h.s.\ is empty), 
\begin{equation}\label{eq-f} 
f\cap (\Gamma+\xi')^\perp \subset \bigcup_{\eta\in\Phi(f)} H(f) +  \eta.
\end{equation} 
\begin{thm} \label{thm-sms}
Let $M$ be an arbitrary model set with polyhedral window $W$ and parameter $\xi = \pim(M)$.
Then 
$$ M \backslash \m_\xi(Int(W)) \subset \bigcup_{f\in \F(W)}  \bigcup_{\eta\in \Phi(f)}
E(f)+\sigma_\xi^{-1}(\eta).$$
\end{thm}
\begin{proof} By Proposition \ref{prop-nonsing} we have 
$$ M \backslash \m_\xi(Int(W)) \subset \bigcup_{f\in \F(W)} \m_\xi(f) .$$
If $A(f)\cap (\Gamma+\xi)^\perp$ is empty then $\m_\xi(f) = \emptyset$. Otherwise
(\ref{eq-f}) implies that 
$$ \m_\xi(f) \subset \bigcup_{\eta\in\Phi(f)} \m_\xi(f_\eta) +\sigma_\xi^{-1}(\eta).$$
where $f_\eta = (f-\eta) \cap H(f)$ is an
open subset of $H(f)$. Finally,
$$\m_\xi(f_\eta) = \{\pi^\|(x):x\in\Gamma, \pi^\perp(x)\in f_\eta\}\subset \sigma^{-1}(H(f)\cap \Gamma^\perp)\subset E(f).$$
\end{proof}
As a first corollary we obtain a generalization of one direction of Proposition \ref{prop-acc} to
singular model sets.
\begin{cor}\label{cor-acc}
  Let $P$ be a $B$-patch for a model set with window $W$ and
  $M\in\Omega(W)$ a possibly singular model set. We assume that $P$
  and $M$ contain the origin.  Let $x\in M$. If $\sigma_{M} (x) \in
  Int(W_P)$ then $P = (M-x)\cap B$.  In particular $\{\sigma_{M}
  (x): P = (M-x)\cap B\}$ is dense in $W_P$.
\end{cor}
\begin{proof} Set $\xi = \pim(M)$.
We see from the description of $M$ given in Theorem \ref{thm-sms} that
the projection set $\m_\xi(Int(W_P))$ is contained in $M$. Hence
condition (\ref{accept}) also applies to a singular $M$,
as long as $\sigma_{M}(0)$ does not lie on the boundary of $W_p$.
Hence if $\sigma_{M}(x) \in Int(W_P)$ then $P = (M-x)\cap B$. 
Denseness of
$\{\sigma_{M} (x): P = (M-x)\cap B\}$ follows directly from that fact that
$\{\sigma_{M} (x): x\in M\}$ is dense in $W$.
\end{proof}
%

For the remainder of this subsection we consider two elements $M_1,M_2$ of the hull of a model set (so with equal cut and project scheme and equal window) which have the same parameter $\xi$. We 
let $$\aaa(M_1,M_2) = \bigcup_{f\in \F_\xi(W)} \bigcup_{\eta\in \Phi(f)}
  E(f)+\sigma_\xi^{-1}(\eta)$$ 
where $\F_\xi(W)\subset \F(W)$ is the set of faces for which $A(f)\cap (\Gamma+\xi)^\perp$ is not empty.  
\begin{cor}
 The difference set $M_1\Delta M_2$
  is contained in $\aaa(M_1,M_2)$. $\aaa(M_1,M_2)$ is
  a finite affine hyperplane arrangement (if non empty).
\end{cor}
\begin{proof}
The first statement follows directly from Theorem \ref{thm-sms}.
  By Lemma~\ref{lem-codim} the sets
  $E(f)+\sigma_\xi^{-1}(\eta)$, $\xi =\pim(M_1)$, are proper affine
  hyperplanes. Indeed, if $\sigma_\xi^{-1}(\eta)$ is not finite then
  it contains a period which must be contained in $E(f)$. 
  
  There
  are only finitely many affine hyperplanes, because $\F(W)$ and $
  \Phi(f)$ are finite.
\end{proof}

The affine hyperplanes making up $\aaa(M_1,M_2)$  depend on the pair $M_1,M_2$.
However, their number is uniformly bounded in the parameter $\xi$. 
\begin{lem}\label{bound-N} 
There is a finite number $N$ such that for all  pairs of model sets $M_1,M_2$
with 
equal parameter, 
the number of hyperplanes in the arrangement $\aaa(M_1,M_2)$ is bounded by $N$.
\end{lem}
\begin{proof}
  There is a finite number of faces and each face $f$ gives rise to a
  set $\Phi(f)$ which can vary with $\xi$, but the number of its
  elements is bounded from above since $f$ is bounded and $\Delta(f)$
  uniformly discrete.
\end{proof}

\begin{prop}\label{rel-dense}
Given $r>0$ there exists a $\rho>0$ such that for all model sets $M_1,M_2$
with 
equal parameter and
all $x\in \R^d$, the ball $B_{\rho}(x)$ contains at least one
point at which $M_1$ and $M_2$ agree out to distance $r$.
\end{prop}

\begin{proof}
  Let $\xi=\pim(M_1)=\pim(M_2)$. Pick a $B$-patch $P$ where $B$
  contains a ball of radius $r$.  Since the interior of the acceptance
  domain $W_P$ is in the interior of $W$, $\m_\xi(W_P^0)$ is a subset
  of both $M_1$ and $M_2$.

  This is a relatively dense subset of
  $M$, 
  so there exists a $\rho$ such that for all $x\in \R^d$ the ball
  $B_{\rho}(x)$ contains at least one point of $\m_\xi(Int(W_P))$, and
  around this point $M_1$ and $M_2$ agree out to distance $r$.
\end{proof}

\subsection{Reprojections} 
In the cut \& project construction described above, 
the projection $\pi^\|: H \times \R^d \to \R^d$ is assumed to be along 
$H$.  If we change the
direction along which we project, 
but otherwise leave the strip $S$ and the parameter $\xi$ fixed, then
this affects the projection set $\m_\xi(W) = \{\pi^\|(x):x\in S\cap
(\Gamma+\xi)\}$ rather mildly.  Let $\pi':H\times\R^d\to \R^d$ be the
projection onto $\R^d$ along another group $H'\subset H\times \R^d$
transverse to $\R^d$.  We call
$$ \m_\xi'(W) = \{\pi'(x):x\in S\cap (\Gamma+\xi)\}$$
the {\em reprojection\/} of $\m_\xi(W)$ along $H'$. 

More generally, if $\Lambda$ is any subset of 
$\pi^\|(\Gamma+\xi)$ 
we call
$$\Lambda' = \{\pi'\circ{\pi^\|_\xi}^{-1}(\lambda):\lambda\in\Lambda\}$$
its reprojection along $H'$. Here $\pi^\|_\xi$ is the
restriction of $\pi^\|$ to $\Gamma+\xi$ which is injective by our
assumption. 

The group $H'$ is the image of $H$ 
under a transformation $g: H \times \R^d \to H \times \R^d$ 
that is the identity on $0 \times \R^d$. (If $H$ is a vector space,
then this is a shear.) 
Since $\pi' = \pi \circ g^{-1}$, 
the reprojection $\m'_\xi(W)$ of $\m_\xi(W)$ 
is itself a model set with internal group $H$, albeit with strip $g^{-1}(S)$,
with lattice $g^{-1}(\Gamma)$, and with parameter $g^{-1}(\xi)$.

\subsection{Pattern equivariant functions}\label{sec-pef}
Let $\Lambda\subset\R^d$ be a FLC-Delone set and $Y$ some set.  A
function $f:\R^d \to Y$ is called {\em \spe\/} if there exists an
$R>0$ (called the radius) such that, whenever $x_1,x_2 \in \R^d$ are
such that $\Lambda - x_1$ and $\Lambda -x_2$ agree exactly on the ball
$B_R(0)$, then $f(x_1)=f(x_2)$. In other words, each function value
$f(x)$ is determined exactly by the pattern of $\Lambda$ in a ball of
radius $R$ around $x$. For most of our purposes we only need functions 
defined on $\Lambda$, or on the $CW$-complex it defines. So we call a
function $\phi: \Lambda \to Y$ \spe\ if there exists an $R>0$ such
that, whenever $x_1,x_2 \in \Lambda$ are such that $\Lambda - x_1$ and
$\Lambda -x_2$ agree exactly on the ball $B_R(0)$, then
$\phi(x_1)=\phi(x_2)$. It can be shown that if $Y$ is a finite
dimensional real vector space then any \spe\ function on $\Lambda$ is
the restriction of a smooth \spe\ function on $\R^d$ \cite{Kel}.

Any locally constant function $\tilde f:\Xi_\Lambda\to Y$ defines a
\spe{} function $f:\Lambda\to Y$ on $\Lambda$ via $f(x) = \tilde f(\Lambda-x)$;
this defines a bijective correspondence.

Now let $Y$ be a metric space.  Any continuous function $\tilde
f:\Omega_\Lambda\to Y$ is uniquely determined by the function
$f:\R^d\to Y$, $f(x) = \tilde f(\Lambda-x)$.  We call a function
$f:\R^d\to Y$ arising in such a way from a continuous function $\tilde
f:\Omega_\Lambda\to Y$ {\em \wpe\/}. Likewise $\phi: \Lambda\to Y$ is
{\em \wpe\/} if it arises in the above way from a corresponding
function $\tilde \phi:\Xi_\Lambda\to Y$. Equivalently we may say that
$\phi: \Lambda\to Y$ is \wpe\ for $\Lambda$ if and only if the function
$\{\Lambda-x : x\in \Lambda\}\ni (\Lambda-x)\mapsto \phi(x)$ is
uniformly continuous in the topology of $\Xi_\Lambda$.  It is not
difficult to see that if $Y=\R^k$ is a finite dimensional vector space
then $\phi$ is \wpe\ if and only if it is the uniform limit of \spe\
functions.

The following is a very important example.  A direct proof not using
the maximal equicontinuous factor map can obtained from Corollary
\ref{cor-W}, see also \cite{BaakeLenz}.

\begin{lem}\label{lem-starwpe}
The star map $\sigma_M:M\to H$ of a model set with metrizable internal group is \wpe.
\end{lem}

\begin{proof}

Recall that the image of $\sigma_M$ lies in $W$. We can hence rewrite (\ref{eq-pistar}) as
$\sigma_M = \imath_H^{-1}\left|_W\right.\circ \pim(M-\cdot)$. 
Since $W\subset H$ is compact 
$\imath_{H}^{-1}\left|_W\right. : [W,0]_\Gamma \to H$ is uniformly continuous. 
Since also $\pim$ is uniformly continuous $\sigma_M$ is \wpe. 
\end{proof}
Let now $Y$ be an abelian group and $\Phi:\Lambda\to Y$. We define 
$\delta\Phi :\Lambda\times\Lambda \to Y$,
\begin{equation}\nonumber
 \delta\Phi(x,y) = \Phi(y)-\Phi(x) 
\end{equation}
and call it the {\em coboundary} of $\Phi$. (In
Section~\ref{sec-cohomology}, we will consider the cohomology of a
CW-complex whose vertex set is $\Lambda$.
In that setting, $\Phi$ is a 0-cochain 
whose coboundary is the restriction of $\delta\Phi$ to a
subset of $ \Lambda\times\Lambda$. This restriction of the domain of
$\delta\Phi$ does not play any role in the sections up to
Section~\ref{sec-cohomology}, so we 
also refer to $\delta\Phi$, with domain $ \Lambda\times\Lambda$, as the
coboundary of $\Phi$.)

A function $\Psi:\Lambda\times\Lambda \to Y$ is called \spe\ if, for
some $R_0$, large enough such that every $R_0/3$-ball contains a point
of $\Lambda$, there exists a radius $R$ such that for all $(x,y)\in
\Lambda\times\Lambda$ with $|y-x|\leq R_0$ the value of $\Psi(x,y)$
depends only on $B_R(0)\cap (\Lambda-x)$ and on $y-x$, i.e.\
$\Psi(x_1,y_1)=\Psi(x_2,y_2)$ whenever $\Lambda-x_1$ and $\Lambda-x_2$ agree
exactly on $B_R(0)$ and $y_2-x_2=y_1-x_1$. This then implies that, for
arbitrary $(x,y)\in \Lambda\times\Lambda$, the value of $\Psi(x,y)$
depends only on $B_{R+|y-x|}(0)\cap (\Lambda-x)$ and $y-x$.

Clearly, if $\Phi:\Lambda\to Y$ is \spe\ then its co-boundary $\delta\Phi$ 
is also \spe.
We are also interested in \wpe\ functions $\Phi:\Lambda\to Y$, 
having values in a metrizable abelian group $Y$, whose coboundaries are 
\spe. 
The star map $\sigma_M$ of a model set with metrizable internal group is a 
key example. By Lemma~\ref{lem-starwpe} $\sigma_M:M\to H$ is \wpe, yet:
\begin{lem}\label{lem-costarspe}
The coboundary of the star map $\sigma_M$ of a model set is \spe.
\end{lem}
\begin{proof} 
Given $x_1,x_2$ in $M$ we have $\delta \sigma_M(x_1,x_2) = \sigma_M(x_2)-\sigma_M(x_1) = (x_2-x_1)^*$, and so depends only on $x_2-x_1$. 
\end{proof}
\begin{remarks}
The above characterisation of the star map of a model set can be understood as a generalization of a result from Boulmezaoud's thesis \cite{B-thesis}. Boulmezaoud
  showed that in the case that the internal group is $H=\R^n$ the star
  map $\R^d\to\R^n$, $x\mapsto x^*$ extends to a \wpe\ smooth function
  whose differential is \spe.  
\end{remarks}

\section{Shape conjugacy and \ns\ generators}
\label{sec-tove}
A shape conjugation is a particular shape deformation in the sense of
\cite{CS2,Kel}, and shape deformations arise as follows:
Consider a function $F:\Lambda\to\R^d$, defined on an FLC Delone set
$\Lambda$.  It defines a new set $$\Lambda^F :=
\{x+F(x):x\in\Lambda\}.$$ 
When the coboundary of $F$ is \spe\ then the elements of $\Lambda^F-\Lambda^F$
can be locally derived from $\Lambda-\Lambda$ and in
particular $\Lambda^F$ has FLC. In this case we call $\Lambda^F$ a
{\em deformation} of $\Lambda$ and 
the function $F$ its {\em generator\/}.

\begin{remarks}
  The term ``deformation of a model set" has been used in the
  literature (see \cite{BaakeLenz,BernuauDuneau}) also for other
  kinds of deformations for which, in particular, the deformed set is
  no longer necessarily FLC. These could be achieved by functions $F$
  whose co-boundaries are not \spe. We have built into the definition of deformation and its generator the requirement that $\delta F$ is \spe\ precisely to guarantee that FLC is preserved under deformation. 
  Without the FLC requirement there
  are many more possible deformations and our rigidity results do not
  apply. 
\end{remarks}
  
\subsection{Shape conjugacy asymptotically negligible generators}  \label{sec-an}
  
 A deformation with generator $F$ is a {\em shape semi-conjugacy\/} if the map
 $\Lambda-x\mapsto \Lambda^F-x$ extends from the orbit of $\Lambda$ in
 $\Omega_\Lambda$ to a topological semi-conjugacy
 $\s_F:\Omega_\Lambda\to\Omega_{\Lambda^F}$. Likewise, 
a deformation for which $\s_F$ is a topological
conjugacy is called a {\em shape conjugacy}.

It turns out that the map $\Lambda-x\mapsto \Lambda^F-x$ extends to a
topological semi-conjugacy $\s_F:\Omega_\Lambda\to\Omega_{\Lambda^F}$
if and only if $F$ is \wpe\ \cite{CS2,Kel}.  Thus, the generator of a
shape semi-conjugacy is a \wpe\ function $F:\Lambda\to\R^d$ whose
coboundary is \spe. We call such generators $F$ {\em asymptotically negligible}, because
they are arbitrarily close to \sP\ generators, and for \sP\ generators 
$\Lambda^F$ can be locally derived from $\Lambda$. 
If $F$ is small enough, then the deformation is invertible in the sense that one can find a generator $G:\Lambda^F\to \R^d$ 
such that $\s_G$ is the inverse of $\s_F$ \cite{Kel}. In this case  $\s_F$ is a shape conjugacy.

Our ultimate aim is to understand the extent to which the dynamical system of a
Delone set determines the Delone set. 
To investigate this question we recall the following theorem. 
\begin{thm}[\cite{KS13}, Theorem 5.1]\label{thm-shape}
  Let $\Lambda$ and $\Lambda'$ be FLC Delone sets that are pointed
  topologically conjugate.  For each $\epsilon >0$ there exists a FLC
  Delone set $\Lambda_\epsilon$ that is MLD with $\Lambda$ and a
  function $F_\epsilon:\Lambda_\epsilon \to \R^d$ whose coboundary is
  \spe\ such that $\Lambda'=\Lambda_\epsilon^{F_\epsilon}$ and
  $\s_{F_\epsilon}:\Omega_{\Lambda_\epsilon} \to \Omega_{\Lambda'}$ is
  a topological conjugacy; 
  in other words $F_\epsilon$ is an asymptotically negligible generator
  of an invertible deformation mapping $\Lambda_\epsilon$ to $\Lambda'$.
  Moreover $\sup_x |F_\epsilon(x)|\leq\epsilon$.
\end{thm}
A specific example of a shape conjugacy for model sets is a reprojection. The goal of this work is to show that for model sets satisfying (H1) and (H2) all shape conjugacies are, up to MLD transformations, of that form.
\begin{cor} \label{cor-repr}
Consider a model set $M$ with metrizable
  internal group $H$. Let
  $L:H\to \R^d$ be a continuous group homomorphism. Then
  $F_L:M\to\R^d$, $F_L(x) = L(\sigma_M(x))$ is a generator of a shape
  conjugation. The model set $M^{F_L}$ resulting from the shape change
  is the reprojection along $H' = \{(h,- L(h))\in H\times\R^d: h\in
  H\}$. 
\end{cor}
\begin{proof}
Lemma~\ref{lem-starwpe} and the
continuity of $L$ imply that $F_L$ is \wpe. Additivity of $L$ and Lemma~\ref{lem-costarspe}
imply that $F_L$ has \spe\ coboundary. Finally, 
we have $M^{F_L} =\{x+{{L}}(\sigma_M(x)):x\in M\} =
\{(\pi^\|+{{L}}\circ\pi^\perp)(\sigma_M(x),x) : x\in M \}$. 
Now the elements in the kernel of $\pi^\|+{{L}}\circ\pi^\perp$ have the form $(h, -L(h))$, $h\in H$ and so $\pi^\|+{{L}}\circ\pi^\perp$ is the
projection onto $\R^d$ along the subspace 
$H' = \{(h,- L(h))\in H\times\R^d: h\in  H\}$. 
\end{proof}

\begin{remarks}
Theorem~\ref{thm-shape} is about pointed topological conjugacy. To
classify hulls of Delone sets up to 
topologically conjugacy we need to understand when, given two
elements $\Lambda_0,\Lambda_0'$ of the same hull $\Omega_\Lambda$, the map
$\Lambda_0-x\to\Lambda_0'-x$ is continuous and 
extends to a topological conjugacy. (We call this 
an {\em auto-conjugacy}.)
A first observation to make is that any auto-conjugacy $\varphi:
\Omega_\Lambda \to \Omega_\Lambda$ must
preserve the equivalence relation given by $\pim$, i.e.\ if
$\pim(\Lambda_1)=\pim(\Lambda_2)$ then
$\pim(\varphi(\Lambda_1))=\pim(\varphi(\Lambda_2))$. Furthermore, the
map induced by $\varphi$ on $\Omega_{max}$ must be the rotation by
$\eta:=\pim(\Lambda_2)-\pim(\Lambda_1)$, since this is the only homeomorphism
on $\Omega_{max}$ mapping $\pim(\Lambda_1)-x$
to $\pim(\Lambda_2)-x$ for all $x\in\R^d$.  It follows then that for
model sets the rotation by $\eta$ must leave the set $NS$ of nonsingular points invariant. 
For most choices of the window 
the only translations leaving $NS$ invariant are the elements of $(\{0\}\times \R^d
+\Gamma)/\Gamma$. In such a case 
the orbit of a non-singular model set $\varphi$ is thus given by 
a global translation. By continuity $\varphi$ must then be everywhere equal to
this global translation. 
To conclude, for generic choices of the window
all auto-conjugacies of a model set are global translations and hence in particular MLD transformations.
The particular cases in which this might not be the case are those in which 
$NS$ admits additional symmetries; they are presently under investigation.
\end{remarks}
\subsection{Nonslip generators}

We will introduce a property for generators of shape conjugations
which characterizes those we have discussed above in the context of model sets. We state the definition for
arbitrary \wpe{} functions, but the primary application is for vector-valued functions.
Recall that $F$ being \wpe\ means that there exists a continuous
function $\tilde F:\Xi_\Lambda\to \R^d$ such that $F(x) = \tilde
F(\Lambda-x)$. We denote $\Rmax^\Xi = \Rmax\cap \Xi\times\Xi$.

\begin{dfn}\label{def-repr}
  Let $\Lambda$ be a Meyer set.  A \wpe{} function $F:\Lambda\to Y$ is
  {\em \ns} if there exists an $\epsilon>0$ such that, 
  for all $(\Lambda_1,\Lambda_2)\in \Rmax^\Xi$ we have
  $\tilde F({\Lambda_1}) = \tilde F({\Lambda_2})$ whenever 
$\Lambda_{1}$ and $\Lambda_2$ agree out to a radius of $\epsilon^{-1}$.
We call $R=\epsilon^{-1}$ the \ns\ radius of $F$.
\end{dfn}
A \spe\ function is manifestly \ns. The star map is not \spe, yet

\begin{lem}
The star map of $M$ is \ns.
\end{lem}

\begin{proof}
Let $M_1,M_2\in\Omega_M$. 
If $\pim(M_1)=\pim(M_2)=\xi$ and $x \in M_1\cap M_2$, then 
equation (\ref{eq-pistar}) implies that
$\sigma_{M_1}(x)=\sigma_{M_2}(x)$. 
Since $\tilde{\sigma_M}(M_i-x) = \sigma_{M_i}(x)$ 
for $i=1,2$ the star map is non-slip
for every positive radius. 
\end{proof}
\begin{cor}
Every generator of a shape conjugation of the form $F_L$ is \ns.
\end{cor}

\subsection{Nonslip = reprojection}

We have just seen that continuous group homomorphisms $L: H \to \R^d$ define \ns\
generators of shape conjugations, and hence that reprojections are generated
by \ns{} generators. We
next show the converse, that under hypotheses H1 and H$2'$, all \ns{} generators are
essentially of this form.  The following theorem is the main result of
this section, and proves half of Theorem~\ref{NewMainTheorem}.

\begin{thm}\label{mainlemma}
 Let $M$ be a model set satisfying hypotheses H1 and H2$\,{}'$.  
Let $F:M\to\R$ be a \ns\ function whose coboundary $\delta F$ is \spe.
Then $F$ can be written as $F(x) = L(\sigma_M(x)) + \psi(x)$, where $L:H\to\R$ is a
  continuous linear map and  $\psi: M\to\R$ is \spe.
\end{thm}

\begin{proof}

To view the \ns{} property from a different angle we consider an
inverse limit construction for the  canonical transversal
$\Xi$.
Let $M_1\sim_\epsilon M_2$ if
$\pim(M_1)=\pim(M_2)$ and the patterns  $M_1$, $M_2$ agree out to radius
$\epsilon^{-1}$.  This is the intersection of two closed equivalence
relations and therefore itself a closed equivalence relation; the
quotient space $\Xi/\sim_\epsilon$ is a compact Hausdorff
space. Denote by $\pi_\epsilon:\Xi\to\Xi/\sim\epsilon$ the canonical
projection. Then
$$ \Xi = \lim_{0 \leftarrow \epsilon} \Xi/\sim_\epsilon$$
and a \wP\ $F$ is \ns\ iff $\tilde F$ is the pullback by $\pi_\epsilon$ of
a continuous function $\tilde F_\epsilon$ on some approximant $\Xi/\sim_\epsilon$. 

Let $F:M\to\R$ be a \ns, \wpe\ function, whose coboundary $\delta F$ is \spe.
Since the star map $\sigma_M$ is also \ns,
there exists an $\epsilon>0$ such that we have the commutative diagram
$$
\begin{array}{rcccc}  
\Xi & \stackrel{\pi_\epsilon}{\to}& \Xi/\sim_\epsilon &\stackrel{\widetilde{\sigma_M}^\epsilon}\to & H \\ 
\tilde F \downarrow & &  \tilde F^\epsilon \downarrow & & \\
\R & = &\R & = &\R 
\end{array}
$$
where 
the continuous maps $\tilde F$ and $\widetilde{\sigma_M}$ are
induced by $F$ and $\sigma_M$. We need to show
that there exists a $0<\eta\leq\epsilon$, a \spe{} function $\psi$ and a continuous group
homomorphism $L$ such that the right side of the diagram can be
completed to a commutative diagram
$$
\begin{array}{rcl}  
\Xi/\sim_\eta &\stackrel{\widetilde{\sigma_M}^\eta}\to & H \\ 
\tilde F^\eta-\tilde\psi^\eta \downarrow & & \downarrow L\\
\R & = &\R 
\end{array}
$$
($\eta^{-1}$ is hence at least as large as the \spe\ radius of $\psi$ and the \ns\ radius of $F$).
Given that a global translation of a point set is an MLD transformation which may be absorbed in the definition of $\psi$ we may assume that $M$ contains the origin.

By FLC there are a finite number of possible $\epsilon^{-1}$-patches
at $0$ which contain the origin. By H$2'$ the acceptance domain of each
$\epsilon^{-1}$-patch can be written as a finite union of closed
convex sets that have non-empty interior.  We call these convex sets
{\em sectors\/} and let $I$ index the sectors of all
$\epsilon^{-1}$-patches.  Thus $\alpha \in I$ denotes both a patch
(located for reference at the origin) and a closed convex subset
$W^\alpha \subset W$ of non-empty interior.  For each such $\alpha$,
let $\Xi^\alpha$ be the corresponding subset of $\Xi$, i.e.\ the set
of point patterns that (a) contain the origin, (b) have the correct
$\epsilon^{-1}$-patch around the origin, and (c) are mapped to sector
$W^\alpha$ by $\widetilde{\sigma_M}$.

The crucial observation is that on $\Xi^\alpha$ the equivalence
relation $\sim_\epsilon$ coincides with the relation $\Rmax$, 
and hence for $M_1,M_2\in \Xi^\alpha$ we have
$M_1\sim_\epsilon M_2$ iff $\widetilde{\sigma_M}(M_1) =
\widetilde{\sigma_M}(M_2)$. In other words, the restriction of
$\widetilde{\sigma_M}^\epsilon$ to $\Xi^\alpha/\sim_\epsilon$ is a
homeomorphism between $\Xi^\alpha/\sim_\epsilon$ and $W^\alpha$.  It
follows that for each $\alpha$ there is a unique function
$f^\alpha:W^\alpha\to\R$ such that
\begin{equation}\label{eq-com}
\begin{array}{rcccl}  
\Xi^\alpha & \stackrel{\pi_\epsilon}{\to}&\Xi^\alpha/\sim_\epsilon &\stackrel{\widetilde{\sigma_M}^\epsilon}\to & W^\alpha \\ 
\tilde F\downarrow & & \tilde F^\epsilon \downarrow & & \downarrow f^\alpha\\
\R& = & \R & = &\R 
\end{array}
\end{equation}
commutes.

Let $y\in D = \sigma^{-1}((W-W)\cap\Gamma^\perp)$ 
be a possible displacement between two points in the same pattern.  
Define $$\Delta_y\tilde F:\Xi \cap (\Xi +y) \to \R$$ 
by
$$\Delta_y\tilde F(M') = \tilde F(M'-y) -\tilde F(M').$$ 
We consider also the
restriction $\Delta_y^{\alpha,\alpha'}\tilde F$ of $\Delta_y\tilde F$
to $\Xi^\alpha\cap (\Xi^{\alpha'}+y)$. Taking into account the fact that
$\widetilde{\sigma_M}(M'-y)- \widetilde{\sigma_M}(M')
=\sigma_{M'}(y) - \sigma_{M'}(0) = y^*$, the preceding
paragraph shows that
$$ \Delta_y^{\alpha,\alpha'}\tilde F(M') = 
\Delta_{y^*}^{\alpha,\alpha'} f(\widetilde{\sigma_M}(M'))
$$
where  $\Delta_{v}^{\alpha,\alpha'} f(u)=f^{\alpha'}(u+v) -   f^\alpha(u)$. 

Since $F$ has \spe{} coboundary, and since \spe{} functions are
locally constant on $\Xi$, $\Delta_y \tilde F$ is locally constant on
$\Xi \cap (\Xi + y)$. Hence by (\ref{eq-com}) the function
$\Delta_{y^*}^{\alpha,\alpha'} f$ is locally constant on $W^\alpha\cap
(W^{\alpha'}-y^*)$. Since $W^\alpha$ and $W^{\alpha'}$ are convex,
$W^\alpha\cap (W^{\alpha'}-y^*)$ is connected or empty. Hence the
function $\Delta_{y^*}^{\alpha,\alpha'} f$ is actually constant on
$W^\alpha\cap (W^{\alpha'}-y^*)$. 

There are finitely many sectors and each sector has non-empty
interior.  So there is an open neighborhood $U\subset H$ of $0$ such
that for all $\alpha$ and all $v\in U$ we have $W^\alpha\cap
(W^{\alpha}-v)\neq\emptyset$.  We claim that for $v\in U$ the value of
$\Delta_{v}^{\alpha,\alpha} f(u)=f^{\alpha}(u+v) - f^\alpha(u)$ is
independent of $\alpha$ as well as independent of $u\in W^\alpha\cap
(W^{\alpha}-v)$ (we know already that it is independent of $u$ if $v=y^*$ with $y\in D$).

The sectors have non-empty interior, so by Corollary \ref{cor-W} for each
sector $\alpha$ we can find a patch $P^\alpha$ of $M$ whose acceptance
domain is contained in the interior of $W^\alpha$.  By repetitivity,
there is a radius such that every ball of that radius contains at
least one copy of each patch $P^\alpha$. Let $P$ be a patch of $M$ of
that radius, so that $P$ contains translates of all the patches
$P^\alpha$.  That is, for each $\alpha$ there is $x^\alpha\in \R^d$ such that
$P^\alpha+x^\alpha$ is a subpatch of $P$. It follows that
$\widetilde{\sigma_M}(M-x^\alpha)$ lies in the interior of $W^\alpha$.
By Corollary \ref{cor-acc} the set $\{\sigma_M(x) : P=(M-x)\cap B\}$ is
dense in the acceptance domain $W_P$.  If $P=(M-x_1)\cap B=(M-x_2)\cap
B$ we call $x_2-x_1$ a {\em return vector\/} of $P$.  The possible
values of $y^*$ for return vectors $y$ of $P$ are thus dense in
$W_P-W_P$.
 
Pick a return vector $y$ of $P$. 
We then have
$$\Delta_y \tilde F(M - x^\alpha) - \Delta_y \tilde F(M - x^{\alpha'})
= \tilde F(M - x^\alpha)-\tilde F(M-x^{\alpha'}) - (\tilde F(M -y-
x^{\alpha})-\tilde F(M-y-x^{\alpha'})) .$$ Now $\tilde F(M -
x^\alpha)-\tilde F(M-x^{\alpha'})$ is obtained by adding the $\delta
F(e_i)$ over the edges $e_i$ along a path in $P$ which joins
$x^\alpha$ to $x^{\alpha'}$, while $F(M-y-x^\alpha) -
F(M-y-x^{\alpha'})$ is obtained by summing the values of $\delta
F(e_i)$ over the corresponding path in $P+y$. Since $y$ is a return
vector to $P$ and $\delta F$ is \spe, the result is the same. Hence
$$\Delta_{y^*}^{\alpha, \alpha}f(\widetilde{\sigma_M}(M-x^\alpha))
- \Delta_{y^*}^{\alpha',
  \alpha'}f(\widetilde{\sigma_M}(M-x^{\alpha'})) = \Delta_y \tilde F(M
- x^\alpha) - \Delta_y \tilde F(M - x^{\alpha'}) = 0.$$ To summarize,
we have established that for all $y^*\in (W_P-W_P)\cap \Gamma^\perp$
the value of $\Delta_{y^*}^{\alpha, \alpha}f(u)$ is the same for all
$\alpha$ and all $u\in W^\alpha\cap (W^\alpha-y^*)$. Moreover, for
fixed $u$ in the interior of $W^\alpha\cap (W^\alpha-y^*)$, the
function $v\mapsto \Delta_{v}^{\alpha, \alpha}f(u)$ is continuous in a
neighborhood of $y^*$.  It follows that $\Delta_{v}^{\alpha, \alpha}
f(u)$ is independent of $\alpha$ and $u\in W^\alpha\cap (W^\alpha-v)$
for all $v \in W_P-W_P$.

Let $\tilde U = (W_P^o - W_P^o)\cap U$.  $\tilde U$ is an open 
neighborhood of the identity in $H$.  
We define $L:\tilde U\to \R$ such that $L(v)$ is the constant value that the
function $\Delta_v^{\alpha,\alpha} f$ takes on
$W^\alpha\cap(W^\alpha-v)$. 
We saw that $L$ is continuous. We claim that $L$ is additive where
sums are defined. Indeed, if $u_1,u_2,u_1+u_2\in U$ then there is $u\in W^\alpha$ such that
also $u+u_1$ and $u+u_1+u_2$ lie in $W^\alpha$.
It follows that, for all  $u\in \tilde U$
$$ L(u_1+u_2) =
f^{\alpha}(u+u_1+u_2) - f^\alpha(u+u_1) + f^\alpha(u+u_1) - f^\alpha(u) = L(u_2)+L(u_1).
$$
A continuous additive function on a neighborhood of the origin is
necessarily linear. That is, $L$  
equals its derivative. 
We may then extend $L$ to a group homomorphism on
the group generated by $\tilde U$ and thus obtain a linear function $L:H\to\R$.

Now let 
$\psi(x) = F(x) - L(\sigma_M(x))$ for $x \in M$. 
This then defines a function $\tilde\psi:\Xi \to \R$,
$\tilde \psi = \tilde F - L\circ\widetilde{\sigma_M}$ which is again
continuous (on $\Xi$).  
If $\widetilde{\sigma_M}(M-x) = \widetilde{\sigma_M}(M-y)\in W^\alpha$ then 
$$\psi(y)-\psi(x) = \Delta_{y-x}^\alpha \tilde F(M-x) - L((y-x)^*) = 0$$
by the construction above and so $\tilde \psi$ is a continuous
function which is constant on $\Xi^\alpha$. Moreover, if $W^\alpha\cap
W^\beta\neq\emptyset$ then continuity implies that $\tilde \psi$ takes
the same value on  
$\Xi^\alpha$ and $\Xi^\beta$. It follows that $\tilde\psi$ is constant on the  
the pre-images 
under $\widetilde{\sigma_M}$ of the connected components of the
acceptance domains of the $\epsilon^{-1}$-patches at $0$. However, different
components of the same central $\epsilon^{-1}$-patch are separated by  
a nonzero distance in $W$, and so 
can be distinguished by the $R$-patches at $0$ for some (possibly large) fixed $R>0$.
Different central patches are distinguished by their
patterns out to distance $\epsilon^{-1}$. Thus $\psi(x)$ is in fact
\spe\ with radius $\eta^{-1}$ where $\eta = \min(R^{-1},\epsilon)$.    
\end{proof}

\begin{cor} If $F$ is a \ns\ generator of a shape conjugation  
for a model set $M$ satisfying H1 and H2$\,{}'$ then, up to MLD transformations, 
$M^F$ is a reprojection of $M$. In particular $M^F$ is a model set.
\end{cor}
\begin{proof}
  Theorem \ref{mainlemma} applied to vector valued functions and
  Corollary \ref{cor-repr} imply that $F$ is the generator of a
  reprojection plus a \spe{} function. Hence, up to an MLD
  transformation, $M^{F}$ is a reprojection of $M$, and is a model
  set.
\end{proof}

\section{Asymptotically negligible = \ns}\label{brillig}

We now turn to the question of when an asymptotically negligible cocycle is \ns. Here we need the stronger assumption H2.

\begin{thm}\label{PolyhedralTheorem}
Let $M$ be a model set satisfying assumptions H1 and H2, and let $F:M\to \R$ be a weakly pattern equivariant function whose coboundary is strongly pattern equivariant. Then $F$ is nonslip.
\end{thm}
\begin{proof} 
Let $F:M\to \R$ be \wpe\ with \spe\ $\delta F$. We may assume that $M$ contains the origin.
We need to show that
there exists $R>0$ such that, for any choice of
pair $M_1,M_2$ in the canonical transversal $\Xi$ with equal parameter $\pim(M_1)=\pim(M_2)$ and any point $x \in M_1$ such that  
and $B_{R}\cap (M_1-x) = B_{R}\cap (M_2-x)$ 
we have $F_2(x)-F_1(x)=0$. Here we have denoted 
$F_i(x) := \tilde F(M_i-x)$ where  $\tilde F$ is the continuous map on the canonical 
transversal that is induced by $F$.

We denote by $R_0$ a radius of pattern equivariance of $\delta F$, that is a radius such that 
$F(x_2)-F(x_1)$ depends only on $B_{R+|x-y|}(0)\cap (\Lambda-x)$ and $y-x$ (see 
Subsection \ref{sec-pef}).
Fix $R>2R_0$ and consider {\em doubly pointed double
  $R$-patches\/}. These are double-$R$-patches $(P,Q)$ 
  of $(M_1,M_2)$ which are centered in a point $z\in M_1\cup M_2$, i.e.\ $P
= M_1\cap B_R(z)$ and $Q = M_2\cap B_R(z)$, together with two points
$x,y\in P\cap Q$ which are at least distance $R_0$ away from $z$ and the boundary $\partial B_R(z)$ of $B_R(z)$.
We denote such an object by
$P^{(2)}(x,y) = (x,y;P,Q;B_R(z))$.
By FLC there are
finitely many up to translation. Since $\delta F$ is \spe\ with radius $R_0$ the
expression 
$$F(P^{(2)}(x,y)):=\delta F_2(x,y)-\delta F_1(x,y) = F_2(y)-F_1(y) -(F_2(x)-F_1(x))$$
depends only on the translational congruence class of $P^{(2)}(x,y)$.
Hence the set $D_R$ of possible values $F$ can take on doubly pointed double $R$-patches is finite. 

We now need the following lemma.

\begin{lem}\label{prop-R_1}
There exists $\tilde N$ and $R_1>0$ such that for all pairs $(M_1,M_2)$ with $\pim(M_1)=\pim(M_2)$
and all 
$x\in M_1\cap M_2$ we have $F_2(x)-F_1(x) \in \tilde D_{R_1}: = {D_{R_1}+\cdots+D_{R_1}}$ ($\tilde N$ copies).
\end{lem}
\begin{proof}[Proof of the lemma]

Let $\V$ be the collection of subspaces $E(f)$ of $\R^d$ associated to the faces of $W$.  
Let $\nu\in\R^d$ be a vector of length one 
and pick $1\geq \omega>0$ such that the cone 
$$C := \bigcup_{\lambda\geq 0} B_{\lambda \omega}(\lambda \nu)$$
intersects each vector space of $\V$ only trivially. 

Now let $M_1,M_2\in \Omega_M$ satisfy $\pim(M_1)=\pim(M_2)$ and
consider a point $x\in M_1\cap M_2$.  Application of
Proposition \ref{rel-dense} with $r>0$ guarantees that there is $\rho>0 $ so
that we may choose a point $x_n\in B_{\rho}(x+n\rho\nu)$ at which $M_1$
and $M_2$ agree, and this for all $n\in\N$. We choose $x_0=x$ and thus
obtain a sequence $\{x_n\}\subset M_1\cap M_2$. 

Let $R_1 = 2\rho + 2R_0$ 
and consider the sequence of doubly pointed double $R_1$-patches
$P^{(2)}_n(x_n,x_{n+1}) =
(x_n,x_{n+1};P_n,Q_n;B_{2\rho}(x+(n+1/2)\rho\nu))$.  If $n\geq
4/\omega$ then $P^{(2)}_n(x_n,x_{n+1})$ is contained in $x+C$. 

Now the  intersection of $x+C$ with a hyperplane from $\aaa(M_1,M_2)$
separates at most one pair of consecutive points $(x_n,x_{n+1})$, if $n\geq
4/\omega$, and therefore intersects at most one double patch $(P_n,Q_n)$.
Thus by Lemma~\ref{bound-N} there are at most
$\tilde N = N+4/\omega$ double patches $(P_n,Q_n)$ which are distinct and so
for at most $\tilde N$ values of $n$ we have $F(P^{(2)}_n)\neq
0$. Once all hyperplanes are crossed, $M_1-x_n$ and $M_2-x_n$ agree
out to distance approximately $n\rho\omega$ and hence $\lim_{n\to
  \infty} (F_2(x_n)-F_1(x_n)) = 0$ by the weak pattern equivariance of
$F$. Thus $F_2(x)-F_1(x) = \sum_n F(P^{(2)}_n) \in \tilde D_{R_1}$
with $R_1 = 2\rho + R_0$.
\end{proof}

We continue the proof of Theorem \ref{PolyhedralTheorem}. If $\tilde
D_{R_1} = \{0\}$, then $F$ is \ns{} with radius $R_1$.  Otherwise, let
$c = \min\{|d|:d\in \tilde D_{R_1}\backslash \{0\}\}$.  Since $F$ is
\wpe\ $\tilde F$ is uniformly continuous, so there exists $R$ such
that $B_{R}\cap (M_1-x) = B_{R}\cap (M_2-x)$ implies $|\tilde
F(M_1-x)-\tilde F(M_2-x)|< c$. But if $B_{R}\cap (M_1-x) = B_{R}\cap
(M_2-x)$ then $x\in M_1\cap M_2$, so by Lemma~\ref{prop-R_1} the 
inequality implies $|\tilde F(M_1-x)-\tilde F(M_2-x)|= 0$.
\end{proof}

\begin{proof}[Proof of Theorem \ref{NewMainTheorem}]

  If $\mo'$ is pointed topologically conjugate to $\mo$, then $\mo'$
  is MLD to a pattern $\mo''$ that is shape conjugate to $\mo$. Let
  $F$ be the generator of that shape conjugacy. By Theorem
  \ref{PolyhedralTheorem} applied to vector-valued functions, $F$ is
  \ns. By Theorem \ref{mainlemma} applied to vector-valued functions,
  $F$ is then the sum of a linear map $L: H \to \R^d$ and a \spe{}
  function $\psi_\epsilon$. Since $L$ induces a reprojection and
  $\psi_\epsilon$ induces an MLD transformation, $\mo''$ is MLD to a
  reprojection of $\mo$.  Since $\mo'$ is MLD to $\mo''$, $\mo'$ is
  also MLD to a reprojection of $\mo$.
\end{proof}

\section{Cohomological interpretation}\label{sec-cohomology}

A Delone set $\Lambda$ is always mutually locally derivable with a
CW-complex $\tilde \Lambda$ in $\R^d$ such that it coincides with the
vertex set of the complex.  Indeed one may take the dual of the
Voronoi complex defined by $\Lambda$ \cite{BoulKel}.  A edge (or
$1$-cell) is then the convex combination of a pair of vertices and
hence the set of edges $\tilde\Lambda^{(1)}$ of the complex is in
bijection with a subset of $\Lambda\times\Lambda$. By construction,
the (necessarily finite) subset of edges emanating from $x$ is locally
derivable from $\Lambda$, that is, determined by the $R$-patch of
$\Lambda$ at $x$, for some fixed $R$. A $0$-cochain on the complex
with values in an abelian group $Y$ is a function $\Phi:\Lambda\to Y$
and a $1$-cochain a function $\Psi:\tilde\Lambda^{(1)}\to Y$. The
usual cohomological definition of coboundary of a $0$-cochain $\Phi$
on that complex is precisely the restriction of $\delta\Phi$ (as
defined in Subsection \ref{sec-pef}) to $\tilde\Lambda^{(1)}$.

In this section we are interested in certain subgroups of the first
cohomology $H^1(\Lambda,Y)$ of $\Lambda$ with values in $Y$ ($Y$ will
be $\Z$, $\R$, or $\R^d$) which is by definition the group of \spe\
$1$-cocycles modulo the coboundaries of \spe\ $0$-cochains. Since the
CW-complex is contractible, every $1$-cocycle is the coboundary of a
$0$-cochain, but not necessarily of a \spe\ one.  The notions of a
\spe\ $0$-cochain, a \wpe\ $0$-cochain, and a \spe\ $1$-cocycle are
therefore those given in Subsection \ref{sec-pef}.

The discussion of Subsection \ref{sec-an} 
(see also \cite{CS2}) can thus be summarized by saying that
shape semi-conjugacies (and small shape conjugacies), up to MLD shape
conjugacies, are parametrized by the sub-group of $H^1(\Lambda,\R^d)$ which consists 
of the classes of \spe\ co-cycles which are coboundaries of \wpe\ functions, as these are precisely the coboundaries of asymptotically negligible generators of deformations. 
We denote this subgroup by $H^1_{an}(\Lambda,\R^d)$.
 
There is an equivalent description of the cohomology $H^*(\Lambda,Y)$ of $\Lambda$
provided that $Y=\R$ or $\R^k$.  We can consider de Rham forms on
$\R^d$ which are \spe\ for $\Lambda$. These form a sub complex of the
usual de Rham complex for $\R^d$ and $H^*(\Lambda,Y)$ can be seen as
the cohomology of this sub complex \cite{Ke1}. In particular, the elements of
$H^1(\Lambda,\R^d)$ can be represented by  \spe\
differentials $dF$ of smooth functions $F:\R^d\to \R^d$ and the
relation with the description by $0$-cochains is given by integration: The map $dF\mapsto \delta F$ where, for a given edge $e=(x,y)\in\tilde\Lambda^{(1)}$, 
$\delta F(e) = \int_e dF = F(y)-F(x)$ defines 
a $1$-cocycle on the CW-complex \cite{BoulKel}.

$H^*(\Lambda,Y)$ is naturally isomorphic to the \v Cech
cohomology $\check H^*(\Omega_\Lambda,Y)$  \cite{integer}. 

There are a number of canonical
subgroups of $H^1(\Lambda, \R^d)$. Our results can be viewed as 
saying when these subgroups are equal and when they are not.

One subgroup, denoted $H^1_{lin}(\Lambda,\R^d)$ is given by generators
$F$ that are the the restriction to $\Lambda$ of linear maps
$L:\R^d\to\R^d$. For each such $F$, the deformed set is simply the
result of applying the linear transformation $\mbox{\rm id}+L$ to the
points of $\Lambda$. If $L$ is non-zero then the deformation cannot be
a local derivation, so $H^1_{lin}(\Lambda,\R^d)$ is isomorphic to
$Hom(\R^d,\R^d) \cong \R^{d^2}$ (as a vector space).

For model sets, reprojections give another subgroup, 
$H^1_{repr}(\Lambda,\R^d)$. Elements of $H^1_{repr}(\Lambda, \R^d)$ 
correspond to generators of the form $F_L(x) = L(\sigma_M(x))$, where
$L:H\to \R^d$ is a continuous group homomorphism. 

Recall that the elements of
$H^1_{an}(\Lambda, \R^d)$ are represented by the \spe\ coboundaries of \wpe\ functions.
Since \wpe\ functions are bounded  (in fact, a generator $F$ is \wpe\ iff it is bounded
  \cite{KS13}) whereas linear maps are unbounded, we must have
$H^1_{lin}(\Lambda,\R^d)\cap H^1_{an}(\Lambda,\R^d)
=\{0\}$.

We also considered \ns{} generators. Recall that any
\spe\ generator is \ns{} and a \ns{} generator is asymptotically
negligible. Hence the classes of coboundaries of \ns\ generators define
a subgroup $H^1_{ns}(\Lambda,\R^d)$ of $H^1_{an}(\Lambda,\R^d)$. 
Furthermore, in the context of  a model set $\mo$, $F_L$ is \ns\ and so
$H^1_{repr}(\mo,\R^d)$ is a subgroup of $H^1_{ns}(\mo,\R^d)$.  
Thus we have a sequence of inclusions
$$ H^1_{repr}(\mo, \R^d) \subset H^1_{ns}(\mo, \R^d) 
\subset H^1_{an}(\mo, \R^d).$$
A natural question is whether these groups coincide.

\subsection{Reinterpretation of Theorems ~\ref{mainlemma} and \ref{PolyhedralTheorem}}
We can also consider functions and cochains 
with values in $\R$ rather than with values in $\R^d$, with
$$ H^1_{repr}(\Lambda, \R^d) =  H^1_{repr}(\Lambda, \R) \otimes \R^d, \quad
H^1_{ns}(\Lambda, \R^d) =  H^1_{ns}(\Lambda, \R) \otimes \R^d, \quad
H^1_{an}(\Lambda, \R^d) =  H^1_{an}(\Lambda, \R) \otimes \R^d.
$$

The following are immediate corollaries of Theorem \ref{mainlemma} 
and Theorem \ref{PolyhedralTheorem}. Together they Theorem~\ref{MainTheorem2}. 
\begin{cor}
If $\mo$ is a model set satisfying
H1 and H\/$2'$, then $H^1_{repr}(\mo, \R) = H^1_{ns}(\mo, \R)$.
\end{cor}
\begin{proof} By Theorem \ref{mainlemma}, each \ns{} generator is
the sum of a linear function on $H$ and a \spe{} function, so 
the cohomology class of its coboundary is in $H^1_{repr}(\mo, \R)$. 
\end{proof}
\begin{cor} 
If $\mo$ is a model set satisfying
H1 and H2, then $H^1_{ns}(\mo, \R) = H^1_{an}(\mo, \R)$.
\end{cor}
\begin{proof} By Theorem \ref{PolyhedralTheorem} an asymptotically negligible generator is \ns.
Taking the cohomology class of its coboundary therefore yields the inclusion 
$H^1_{ns}(\mo, \R) \subset H^1_{an}(\mo, \R)$.
\end{proof}
\subsection{Image of the first cohomology of the maximal
  equicontinuous torus}
The maximal equicontinuous factor map $\pi_{max}:\Omega\to
\Omega_{max}$ induces an injective map in cohomology 
$\pi_{max}^*:H^1(\Omega_{max},\Z)\to H^1(\Omega,\Z)$. 
We have thus a fourth subgroup which is worth comparing with the other, 
namely the image under $\pim^*$ of  $H^1(\Omega_{max},\R^d)$ which we denote by 
$H^1_{max}(\Lambda,\R^d)$. 
To do that we need a better understanding of
the image of $\pi_{max}^*$.

We recall from \cite{BKS} that $\Omega_{max}$ can be alternatively described with the
help of the topological eigenvalues of the action. Let $\hat \R^d =
Hom(\R^d,U(1))$ be the Pontryagin dual of $\R^d$, where we require
each homomorphism to be continuous on $\R^d$. An element
$\chi\in\hat{\R^d}$ is a topological eigenvalue if there exists a
non-vanishing continuous function $f:\Omega\to\C$ such that
$f(\Lambda-t) = \chi(t)f(\Lambda)$. Topological eigenvalues form a
countable subgroup $\eee$ of $\hat{\R^d}$ and $\Omega_{max}$ can be
identified with the dual $\hat\eee = Hom(\eee, U(1))$, where $\eee$ is
given the discrete topology. 

For a general topological space $X$,
$H^1(X,\Z)$ is isomorphic to $[X,S^1]$, the homotopy classes of
continuous maps $X\to S^1$. Furthermore, if $\varphi:X\to Y$ is a
continuous map then $\varphi^*:H^1(Y,\Z)\to H^1(X,\Z)$ can be
identified with the mapping $[Y,S^1]\ni[f] \mapsto [f\circ\varphi]\in
[X,S^1]$. We apply this to $\pi_{max}:\Omega\to\Omega_{max}$. 
Since $\Omega_{max}\cong\hat{\mathcal E}$ the elements of
$[\Omega_{max},S^1]$ are the homotopy classes of characters on
$\hat{\mathcal E}$. Hence  $[\Omega_{max},S^1]\cong \mathcal E$ and 
the image of $\chi\in \mathcal E$ under $\pi_{max}^*$ in $[\Omega,S^1]$
is given by the homotopy class of an eigenfunction $f_\chi$ of $\chi$ ($f_\chi$ is
normalized so as to have modulus $1$). This describes the image of
$\pi_{max}^*$ (in degree one) in $[\Omega,S^1]$ \cite{BKS}. 

To obtain the image of $\pi_{max}^*$ in pattern equivariant cohomology
we consider the restriction of a representative $f$ of an
element of $[\Omega,S^1]$ to the orbit of $\Lambda$ and define 
$$\check f:\R^d\to S^1, \quad \check f(x) := f(\Lambda-x).$$ 
\begin{lem}
Any element of $[\Omega,S^1]$ admits a representative  $f$ such
that
$\check f(x) := f(\Lambda-x)$ is \spe.
\end{lem} 
\begin{proof}
$[\Omega,S^1]$ can be seen as the direct limit of $[\mathcal G_n,S^1]$ where
$\mathcal G_n$ is the $nth$ approximant in the G\"ahler complex \cite{SadunG}.
Each element thus comes from some $[\mathcal G_n,S^1]$ and the latter
elements produce \spe\ functions when considered on the orbit.  
\end{proof}
Since $\R^d$ is simply connected we can
lift $\check f$ to a continuous function
$\tau:\R^d\to\R$ such that $\check f(x) = \exp{2\pi i \tau(x)}$.
We define $F$ to be the restriction of $\tau$ to $\Lambda$. Then $\delta
F$ is \spe\ and so we have a map
$[\Omega,S^1]\to H^1(\Lambda,\Z)$: $[f] \mapsto [\delta F]$.
\begin{lem}
  With the above notation $[f] \mapsto [\delta F]$ is a group
  homomorphism whose image corresponds to the image of $[\Omega,S^1]$
  in $H^1(\Omega,\R)$ under the identification $H^1(\Omega,\Z)\cong
  H^1(\Lambda,\Z)\subset H^1(\Lambda,\R)$.
\end{lem} 
\begin{proof}
  Using G\"ahler's approximation this statement boils down to consider
  the map between $[\mathcal G_n,S^1]$ and $H^1(\mathcal G_n,\R)$
  where the latter can be considered as de Rham cohomology on a
  branched manifold. If $[f]\in [\mathcal G_n,S^1]$ then (assuming
  without restriction of generality that $f$ is smooth) $\frac1{2\pi
    i}f^{-1}df$ represents the element in $H^1(\mathcal G_n,\R)$ under
  the map $[\mathcal G_n,S^1]\to H^1(\mathcal G_n,\Z)\subset
  H^1(\mathcal G_n,\R)$. This is well-known (see also \cite{KP}). Now
  in order to obtain the map on cellular cohomology of $\mathcal G_n$
  one just needs to integrate over $1$-chains. The result is a
  $1$-cochain which, when interpreted as \spe\ $1$-cochain on
  $\Lambda$ coincides presicely with $\delta F$.
\end{proof}

If we combine the two arguments we see that the image of
$H^1(\Omega_{max},\Z)\cong \mathcal E$ in pattern equivariant
cohomology can be described as follows.
\begin{cor}\label{cor-H1}
Upon the above identification of $H^1(\Omega_{max},\Z)$ with $\eee$ and the identification of 
$H^1(\Omega,\Z)$ with $H^1(\Lambda,\Z)$ the map $\pim^*$ becomes
the map 
$$\eee\ni \chi \mapsto [\delta\tilde\beta\left|_\Lambda\right. ]\in H^1(\Lambda,\Z)$$
where $\tilde \beta\left|_\Lambda\right.$ is the restriction to $\Lambda$ of a
continuous function $\tilde \beta :\R^d\to \R$ such that $\exp(2\pi i\tilde\beta)$ is \spe\ and
homotopic to $\chi$.
\end{cor}
Each eigenvalue $\chi$ can be lifted, that is 
there exists a $\beta\in{\R^d}^*$ such that $\chi(x) = \exp 2\pi i \beta(x)$. Therefore
$\tilde \beta - \beta$ must be bounded and hence \wpe. It follows
that $[\delta\tilde\beta] - [\delta\beta]\in
H^1_{an}(\Lambda,\R)$ and, since $\beta$ is linear, we see that
$$ H^1_{max}(\Lambda,\R^d) \subset
H^1_{lin}(\Lambda,\R^d)+H^1_{an}(\Lambda,\R^d).$$
Stated differently, the image of $ H^1_{max}(\Lambda,\R^d)$ in the
quotient group\footnote{This quotient group
is the mixed group of \cite{Kel} studied in \cite{B-thesis} for projection method tilings.}
 $$H^1_{m}(\Lambda,\R^d) :=
H^1(\Lambda,\R^d)/H^1_{an}(\Lambda,\R^d)$$
can be identified with a subspace of the vector space of linear deformations.
Indeed let 
$\psi: H^1(\Omega_{max},\R^d) \to
H^1_{m}(\Lambda,\R^d)$ be the composition of $\pi_{max}^*$ with the
canonical projection. It is induced by 
$[\exp 2\pi i \beta] \mapsto [\delta\beta]$. 
\begin{lem}
Let $\Lambda$ be a repetitive FLC Delone set.
We have
$$\psi(H^1(\Omega_{max},\R^d)) =
\big(H^1_{lin}(\Lambda,\R^d)+H^1_{an}(\Lambda,\R^d)\big)/
H^1_{an}(\Lambda,\R^d)\cong Hom(\R^d,\R^d)$$ 
whenever $\Lambda$ is topologically conjugate to a Meyer set.
\end{lem}
\begin{proof}
The above shows that $\psi(H^1(\Omega_{max},\R^d)) \subset
\big(H^1_{lin}(\Lambda,\R^d)+H^1_{an}(\Lambda,\R^d)\big)/
H^1_{an}(\Lambda,\R^d)$ and the r.h.s.\ is clearly isomorphic to the vector space of all linear deformations. Thus  equality holds precisely if the real span of $\{\beta\in{\R^d}^* :\exp 2\pi i \beta \in \eee\}$ has dimension $d$. By the results of \cite{KS13} this is equivalent to saying that
$\Lambda$ is topologically conjugate to a Meyer set.
\end{proof}
\subsection{Case of model sets}
In the case of model sets which satisfy H1 we can say more, because we
have a more explicit model for the maximal equicontinuous factor,
namely the torus parametrization.
Indeed, if $H = \R^n$ then, by cocompactness of $\Gamma$,   
$\Omega_{max} =  H\times\R^d/\Gamma$ is an $n+d$-torus 
and so we can identify $H^1(\Omega_{max},\Z)\cong \eee\cong \hat\Omega_{max}$ 
with the
so-called reciprocal lattice  $\Gamma^{rec}$ which
is given by those linear maps $\alpha:\R^d\times\R^d\to \R$ which satisfy 
$\alpha(\gamma) \in\Z$ for all $\gamma\in\Gamma$. 
\begin{prop} Let $\mo$ be a model set satisfying H1.
Upon the identification of $H^1(\Omega_{max},\Z)$ with $\Gamma^{rec}$ and the identification of 
$H^1(\Omega,\Z)$ with $H^1(\mo,\Z)$ the map $\pim^*$ becomes
the map 
$$\Gamma^{rec}\ni \alpha \mapsto [\delta \alpha(\sigma_M(\cdot),\cdot))]\in H^1(\mo,\Z).$$
\end{prop}
\begin{proof}
The eigenvalue $\chi\in\eee$ corresponding to $\alpha\in\Gamma^{rec}$ is  
$\chi(x) = \exp 2\pi i \alpha(0,x)$, $x\in\R^d$. Recall that
$\sigma_M:M\to H$ is \wpe\ and has \spe\ coboundary. We may therefore extend it to
a \wpe\ function on all of $\R^d$ whose differential is \spe\
\cite{Kel}.  We denote the extension also by $\sigma_M$.  
As  $\{\sigma_M(x):x\in M\}$ lies in a
compact subset of $H$, 
$\chi_t(x) = \exp 2\pi i \alpha(t\sigma_M(x),x)$
is a homotopy between $\chi$ and $\chi_1$.
Furthermore, if $x,y\in M$ then $(\sigma_M(y),y)-(\sigma_M(x),x)=
((y-x)^*,y-x)\in\Gamma$ and so $M\ni x \mapsto \exp 2\pi i
\alpha(\sigma_M(x),x))$ is constant and hence a \spe\ function on
$M$. It follows that $\chi_1$ is a \spe\ function on $\R^d$ and thus
we may apply Corollary \ref{cor-H1} to obtain the statement.
\end{proof}

\begin{prop}\label{prop-impi} 
Let $\mo$ be a model set with internal group $H=\R^n$.
Then $H^1_{repr}(\mo,\R)$ has dimension $n$ and
$$H^1_{max}(\mo,\R^d) =  H^1_{repr}(\mo,\R^d)\oplus H^1_{lin}(\mo,\R^d).$$
\end{prop}
\begin{proof}
We have $H^1_{max}(\mo,\R)\cong\Gamma^{rec}\otimes_\Z\R\cong (\R^n\times\R^d)^*\cong
{\R^n}^*\oplus{\R^d}^*$.  Thus an element corresponding to $\alpha\in\Gamma^{rec}\otimes_\Z\R$ can be split into $(\alpha^\perp,\alpha^\|)\in {\R^n}^*\oplus{\R^d}^*$. Under this splitting the coboundary
$\delta \alpha(\sigma_M(\cdot),\cdot)$ becomes $(\delta \alpha^\perp\circ\sigma_M,\delta \alpha^\|)$.
With $\R^d$ coefficients this induces exactly the splitting $H^1_{max}(\mo,\R^d) =  H^1_{repr}(\mo,\R^d)\oplus H^1_{lin}(\mo,\R^d)$. This also shows that the dimension of $H^1_{repr}(\mo,\R)$ is $n$.
\end{proof}
With this proposition at hand we see that Theorem~\ref{MainTheorem2} is equivalent to
Theorem~\ref{MainTheorem3}. Indeed, it shows that for polyhedral model sets 
the statement $\dim H^1_{an}(\mo,\R)=n$ is equivalent to  
$H^1_{an}(\mo,\R) \subset H^1_{max}(\mo,\R)$.

\section{Nonslip sets and the Meyer property}\label{sec-Meyer}

In this section we explore a little further the concept of \ns\
generators.  Our hope is that this may turn out useful later for the
study of shape conjugations of Delone sets which are not model.  We
consider an analogous property of sets, and show that a generator of a
shape conjugation is \ns\ if an only its associated shape conjugacy
preserves that property.

\begin{dfn} A Delone set $\Lambda$ is {\em \ns} if for all $R>0$ one
 can find an $\epsilon>0$ such that for all $\Lambda_1,\Lambda_2\in\Omega_\Lambda$
 we have that if
  $\pim(\Lambda_1)=\pim(\Lambda_2)$ and $d(\Lambda_1,\Lambda_2)\leq
  \epsilon$, then both sets agree on $B_R(0)$.
\end{dfn}

\begin{lem}
  Every Meyer set is \ns. 
\end{lem}
\begin{proof}
Recall that for Meyer sets $\pim(\Lambda_1)=\pim(\Lambda_2)$ implies that $0\in \Lambda_1-\Lambda_2$. But $\Lambda_1-\Lambda_2$
  is uniformly discrete by the Meyer property. Hence if $\Lambda_1$
  and $\Lambda_2$ are close enough they have to coincide on a ball of
  radius equal to the inverse of their distance.
\end{proof}

 Recall that a  generator of shape conjugacy of an FLC Delone set $\Lambda\subset\R^d$ is a 
 function $F:\Lambda\to\R^d$ such that $\Lambda\mapsto \Lambda^F = \{x+F(x):x\in\Lambda\}$ extends to an $\R^d$-equivariant homeomorphixm $\s_F:\Omega_\Lambda\to \Omega_{\Lambda^F}$. 
We saw that in this case $F$ extends to a continuous map $\tilde F:\Omega_\Lambda\to\R^d$ and so we may define
$ F_{\Lambda'}:\Lambda'\to\R^d$
by $F_{\Lambda'}(x) = \tilde F(\Lambda'-x)$. It is not difficult to see that $ F_{\Lambda'}$ is \wpe\ 
for ${\Lambda'}$ 
and that $\s_F(\Lambda') = {\Lambda'}^{F_{\Lambda'}} = \{x+F_{\Lambda'}(x): x\in {\Lambda'}\}$.
Indeed $\s_F(\Lambda') = \lim \s_F(\Lambda)-x_n$ for some sequence $(\Lambda-x_n)_n$ converging to $\Lambda'$, and 
$\s_F(\Lambda)-x_n = \{x+\tilde F(\Lambda - x):x\in\Lambda\}-x_n =  \{y+\tilde F(\Lambda -x_n-y):y\in\Lambda-x_n\}$. Now since $\tilde F$ is bounded and continuous we conclude that
$\lim \s_F(\Lambda)-x_n = \{y-\tilde F(\Lambda'-y):y\in\Lambda'\}$.

Nonslip generators of shape conjugacies and \ns\ sets are closely
related.

For a more general (not necessarily Meyer) Delone set we
generalize the concept of a \ns\ \wpe\ function as follows:
\begin{dfn}
Let $\Lambda$ be an FLC Delone set.
A  \wpe{} function
$F:\Lambda\to Y$ is {\em \ns\/} if there exists an $\epsilon>0$ such that for all $\Lambda_1,\Lambda_2\in\Rmax^\Xi$ we have
$\tilde F(\Lambda_1) = \tilde F(\Lambda_2)$ whenever $d(\Lambda_1,\Lambda_2)\leq \epsilon$.
\end{dfn}

Note that if $\Lambda$ is \ns, then the above definition reduces
to the definition we previously gave for Meyer sets. This
follows as in the proof of the last lemma form the fact  that
$\pim(\Lambda_1)=\pim(\Lambda_2)$ implies that $\Lambda_1$ and $\Lambda_2$
 agree on balls once they are close.
\begin{prop}
Let $\Lambda$ be \ns. $\Lambda^F$ is \ns\ iff $F$ is \ns.
\end{prop}
\begin{proof}
\noindent ``$\Rightarrow$'' We suppose that $\Lambda^F$ is \ns. Hence, given $R$
there exists $\delta$ such that for all $(\Lambda_1,\Lambda_2)\in
\Rmax$ we have $d(\s_F(\Lambda_1), \s_F(\Lambda_2))<\delta$ implies
$B_{R}[\s_F(\Lambda_1)] = B_{R}[\s_F(\Lambda_2)]$. Moreover, $\tilde F$ is
uniformly continuous so there exists $\epsilon_1$ such that for all $\Lambda_1,\Lambda_2\in
\Xi$ we have $d(\Lambda_1,\Lambda_2)<\epsilon_1$ implies
$\|\tilde F(\Lambda_1) - \tilde F(\Lambda_2)\|<R^{-1}$. Furthermore, $\s_F$
is uniformly continuous so  there exists $\epsilon_2$ such that for
all $\Lambda_1,\Lambda_2\in
\Omega$ we have $d(\Lambda_1,\Lambda_2)<\epsilon_2$ implies
$d(\s_F(\Lambda_1), \s_F(\Lambda_2))<\delta$. Finally, $\Lambda$ is \ns\
so
there exists $\epsilon_3$ such that for all $(\Lambda_1,\Lambda_2)\in
\Rmax$ we have $d(\Lambda_1,\Lambda_2)<\epsilon_3$ implies
$B_{1}[\Lambda_1] = B_{1}[\Lambda_2]$. Let $\epsilon =
\min\{\epsilon_1,\epsilon_2,\epsilon_3,\}$ (which depends on
$R$). Then for all $(\Lambda_1,\Lambda_2)\in
\Rmax^\Xi$ we have $d(\Lambda_1,\Lambda_2)<\epsilon$ implies that
$0\in \Lambda_1\cap\Lambda_2$, $B_{R}[\s_F(\Lambda_1)] =
B_{R}[\s_F(\Lambda_2)]$ and $\|F_{\Lambda_1}(0) - F_{\Lambda_2}(0)\|<R^{-1}$.
Since $\s_F(\Lambda_i) = \{ x + F_{\Lambda_i}(x) : x\in\Lambda_i\}$ we
see that, if $R$ is large enough, this implies $F_{\Lambda_1}(0) = F_{\Lambda_2}(0)$.

\medskip

\noindent ``$\Leftarrow$'' We suppose that $F$ is \ns, hence there exists $\delta$ such that
for all $(\Lambda_1,\Lambda_2)\in\Rmax^\Xi$ and all $x\in\Lambda_1\cap\Lambda_2$ we have that $d(\Lambda_1-x,\Lambda_2-x)<\delta$ implies $F_{\Lambda_1}(x) = F_{\Lambda_2}(x)$.
By definition of the metric there exists $\epsilon_1$ such that if $d(\Lambda_1,\Lambda_2)<\epsilon_1$ then  $d(\Lambda_1-x,\Lambda_2-x)<\delta$ for all $x$ of size smaller or equal to the radius of relative denseness.
Let $R>0$ and $\|F\|=\sup_{\Lambda'\in\Omega} \|\tilde F(\Lambda')\|$ which is finite, by the continuity of $\tilde F$. 
Since $\Lambda$ is \ns\ there exists $\epsilon_2$ such that
for all $(\Lambda_1,\Lambda_2)\in\Rmax$ we have $d(\Lambda_1,\Lambda_2)<\epsilon_2$ implies
$B_{R+\|F\|}[\Lambda_1] = B_{R+\|F\|}[\Lambda_2]$. Let $\epsilon = \min\{\epsilon_1,\epsilon_2\}$.
Then, if $(\Lambda_1,\Lambda_2)\in\Rmax$ and $d(\Lambda_1,\Lambda_2)<\epsilon$ 
we have $B_{R+\|F\|}[\Lambda_1] = B_{R+\|F\|}[\Lambda_2]$ and 
 $F_{\Lambda_1}(x) =
F_{\Lambda_2}(x)$ for all $x\in \Lambda_1\cap \Lambda_2$ of size smaller or equal to the radius of relative denseness. It follows that $B_{R}[\s_F(\Lambda_1)] = B_{R}[\s_F(\Lambda_2)]$. 
Thus $\Lambda^F$ is \ns.
\end{proof}

The following corollary is just a special case:
\begin{cor}\label{cor-ns}
Let $\Lambda$ be a Meyer set and $F$ be a generator of a shape conjugacy.
If $\Lambda^F$  is a Meyer set then $F$ must be \ns.
\end{cor}
The contrapositive says that if $F$ is not \ns, then $\Lambda^F$ is not
Meyer. That is 
 Theorem \ref{MainTheorem4}.

Applying these observations to model sets we obtain:
\begin{cor}\label{cor-last}
Given a model set $\mo$ which satisfies H1.
If $H^1_{ns}(\mo,\R) \subset H^1_{max}(\mo,\R)$
then any shape conjugation of $\mo$ which is a Meyer set is a reprojection of $\mo$.
\end{cor}
\begin{proof}
Let $F$ be the generator of a shape conjugation of $\mo$ such that $\mo^F$ is a Meyer set.
By Corollary \ref{cor-ns}, $F$ must be \ns.
By Proposition \ref{prop-impi}, $H^1_{ns}(\mo,\R) \subset H^1_{max}(\mo,\R)$ 
is equivalent to the equality
$H^1_{ns}(\mo,\R) = H^1_{repr}(\mo,\R)$. Hence $ \mo^F$ is a reprojection.
\end{proof}


\section{A model set that is not \brillig}\label{examples}

Theorem \ref{NewMainTheorem} states that most common examples of model sets,
constructed by direct application of the cut \& project method with polyhedral windows, are \brillig.  
In this example we exhibit a model set with a Euclidean internal group (so satisfying H1) that is
not. The example is constructed by a substitution;  model sets arising from substitutions may have very complicated (fractal) windows.

Consider the 1-dimensional substitution $\sigma$ on four letters, a sort of doubling of the Fibonacci substitution: 
\begin{eqnarray*} \sigma(a_1) & = &  a_1b_1a_2 \\
\sigma(b_1) & = &  a_1b_2 \\
\sigma(a_2) & = &  a_1b_2a_2 \\ 
\sigma(b_2) & =  & a_2b_1.
\end{eqnarray*}
Its substitution matrix is
$$\sm = \begin{pmatrix} 1&1&1&0 \cr 1&0&0&1 \cr 1&0&1&1 \cr 0&1&1&0 \end{pmatrix}$$
and has eigenvalues $\phi^2$, $-\phi$, $\phi^{-1}$ and $\phi^{-2}$
where $\phi = (1+\sqrt{5})/2$ is the golden mean.  We choose tile
length proportional to the right Perron Frobenius eigenvalue, namely
$b_1$ and $b_2$ tiles to have unit length, and $a_1$ and $a_2$ tiles
to have length equal to $\phi =(1+\sqrt{5})/2$.  This is then a
geometric primitive aperiodic unimodular Pisot substitution (which is
not irreducible).
By the results of \cite{BSW} the dynamical spectrum is pure point,
indeed the balanced pair $(a_1b_1,b_1a_1)$ terminates with coincidence
(see \cite{BSW} for explanations on this notion).  A lot is known
about such substitution tilings. The following can be found more or less implicit in, for instance,
\cite{BM,BBK,Sing}.
\begin{itemize}
\item Since the spectrum is pure point, the tilings are MLD to
  (possibly colored) regular model sets. 
 Moreover,  
  the set $M$ of left boundary points of $a_1$-tiles in a tiling is MLD
  to the tiling and hence also a regular model set.
\item Since the substitution is unimodular, the maximal
  equicontinuous factor $\Omega_{max}$ of the associated dynamical system
  is a torus of dimension $J$ where $J$ is the algebraic
degree of the Perron Frobenius eigenvalue.
This eigenvalue is here $\phi^2$ and hence $J=2$.  
This, in turn implies that $M$ has a
  cut \& project scheme in which the internal group $H$ is $\R$. 
\end{itemize}
One readily computes, e.g., using the technique of \cite{AP}, that 
$H^1(\Omega_M, \R) = \R^4$ and that substitution acts on $H^1(\Omega_M,\R)$ by
the transpose $\sm^T$ of the substitution matrix.  For substitution tilings,
$H^1_{an}(\Omega_M, \R)$ is the span of all of the generalized
eigenspaces of this action
with eigenvalues strictly inside the unit circle \cite{CS2}. In our case
this means that $H^1_{an}(\Omega_M, \R)=\R^2$ is the span of the
$\phi^{-1}$ and $\phi^{-2}$ eigenvectors of $\sm^T$. 

The generator of a shape conjugacy 
corresponding to the $\phi^{-2}$ eigenvector induces a reprojection and hence is
\ns. It corresponds to a shape conjugacy in which all of the $a$
tiles are lengthened and the $b$ tiles are shortened (or vice-versa),
while maintaining $|a_1|=|a_2|$ and $|b_1|=|b_2|$ and preserving the quantity 
$|a_1|\phi + |b_1| = \phi^2+1$.

The generator $F$ of a shape conjugacy corresponding to the $\phi^{-1}$ eigenvector is
not \ns, and results in a Delone set $\mo^F$ that is not Meyer.  To
see this, let 
$A_i^n = \sigma^n(a_i)$ and $B_i^n = \sigma^n(b_i)$
be $n$-th order supertiles, and let $|A_i|^F$ and $|B_i|^F$ be the
Euclidean lengths of these supertiles after deformation. This time
$|A_1|^F\neq |A_2|^F$; in fact, $|A_1^n|^F-|A_2^n|^F$ is
proportional to $\phi^{-n}$. 
We have $|A_1^n|^F \in \mo^F-\mo^F$, since
$A_1^nB_1^n$ appears in our tiling and $A_1^n$ and $B_1^n$ start with $a_1$.
Likewise, $|A_2^n|^F \in \mo^F-\mo^F$, since $A_2^nB_1^n$ appears in the tiling and also $A_2^n$ starts with $a_1$. However, since 
$|A_1^n|^F-|A_2^n|^F$ is proportional to $\phi^{-n}$ the set of differences
$\mo^F-\mo^F$ cannot be uniformly discrete, so $\mo^F$ is not
Meyer.

Moreover, we claim that $F$ is not \ns. To see this, note that erasing the subscripts
of all the tiles gives a factor map to the Fibonacci tiling space. 
Furthermore, the maximal equicontinuous factor map of the doubled Fibonacci tiling factors through the above factor map. Hence if $\Lambda_{1},\Lambda_{2}$ 
map to the same Fibonacci tiling, then $\pi_{max}(\Lambda_1)=\pi_{max}(\Lambda_2)$.
As they are model sets this implies that 
$\Lambda_1$ and $\Lambda_2$ will agree exactly on 
arbitrarily large regions \cite{BargeKellendonk}.
If $F$ were \ns, then $F$ would have to take on the same values on
corresponding regions of $\Lambda_1$ and $\Lambda_2$. It follows that for all sufficiently large regions of 
$\Lambda_{1,2}$ that agree exactly, the corresponding regions of  
$\Lambda^F_1$ and $\Lambda^F_2$ must also agree exactly.

However, we will exhibit such a pair $\Lambda_{1,2}$, in which there are arbitrarily large regions where $\Lambda_1$ and 
$\Lambda_2$ agree but $\Lambda_1^F$ and $\Lambda_2^F$ do not. 
As a step to constructing $\Lambda_{1,2}$, 
note that $B_1^2 = a_1 b_1 a_2a_2 b_1$ and $B_2^2=a_1b_2a_2a_1b_2$ both begin with
$a_1$ and then differ, with $B_1^2$ containing a $b_1$ where $A_2^2$ contains a $b_2$, after which both words continue
with $a_2$. Applying the substitution $2n$ times, placing the origin where $B_1^{2n}$ and $B_2^{2n}$  start to differ, and then taking a limit as $n \to \infty$, we obtain two 
different (colored) point patterns $\Lambda_1$ and $\Lambda_2$ that agree exactly on the left half-line $(-\infty,0)$ 
but do not agree on the tile immediately to the right of the origin. 
In  $\Lambda_1$, for each $n$ the origin lies inside a $B_1^{2n}$ that is followed by 
an $A_2^{2n}$, while in $\Lambda_2$, the origin lies in a $B_2^{2n}$ followed by an $A_2^{2n}$. Prior to 
the deformation we have  $|B_1^{2n}| = 
|B_2^{2n}|$, so $\Lambda_1$ and $\Lambda_2$ agree on arbitrarily large regions of the right half-line (as well as on all of the 
left half-line). Moreover, $\Lambda_1$ and $\Lambda_2$ both correspond to the same Fibonacci tiling, so 
$\pi_{max}(\Lambda_1)=\pi_{max}(\Lambda_2)$. 

Now consider $\Lambda^F_1$ versus $\Lambda^F_2$. The deformation makes $|B_1^{2n}|^F$ different from  
$|B_2^{2n}|^F$, so either $\Lambda_1^F$ and $\Lambda_2^F$  disagree on the left half-line, or they disagree on the 
$A_2^{2n}$ supertile following the $B_i^{2n}$ that contain the origin. Either way, the deformation must shift arbitrarily large
corresponding regions relative to one another, and so cannot be \ns.
  
The key feature of this example is that the substitution
matrix is reducible. The maximal equicontinuous factor is determined
by the dynamical spectrum, which for self-similar tile lengths is
determined by the Perron-Frobenius eigenvalue $\lambda_{PF}$. If
$\lambda_{PF}$ is a Pisot number, then a basis for the 
generators of reprojections is given by the eigenvectors with
eigenvalues algebraically conjugate to $\lambda_{PF}$. 
However, a
basis for $H^1_{an}$ is given by {\em all\/} the eigenvectors with
eigenvalue strictly smaller than 1. The eigenvectors whose small
eigenvalues are not conjugate to $\lambda_{PF}$ correspond to shape 
conjugacies that are not reprojections.

In general, whenever a 1-dimensional 
substitution has a small eigenvalue that is not conjugate to 
$\Lambda_{PF}$, the shape conjugacy corresponding to 
that small eigenvalue will not be a reprojection, and will
destroy the Meyer property. We expect the generators of these
shape conjugacies to not be \ns, as in our example, 
but this has not been proven in general.
 
\medskip

{\bf Acknowledgments}: We thank Marcy Barge for showing us the doubled
Fibonacci substitution that we employed in Section \ref{examples}.
The work of the first author is partially supported by the ANR SubTile
NT09 564112.  The work of the second author is partially supported by
NSF grant DMS-1101326.

\end{document}